\documentclass{amsart}
\usepackage{graphicx, amsmath,amsfonts, amssymb}
\usepackage{epstopdf}
\newtheorem{thm}{Theorem}[section]
\newtheorem{cor}[thm]{Corollary}
\newtheorem{lemma}[thm]{Lemma}
\newtheorem{prop}[thm]{Proposition}
\theoremstyle{definition}
\newtheorem{defn}[thm]{Definition}
\theoremstyle{remark}
\newtheorem{rem}[thm]{Remark}
\newtheorem{remark}[thm]{Remark}
\numberwithin{equation}{section}
\newcommand{\abs}[1]{\left\vert#1\right\vert}

\newcommand{\Real}{\mathbb R}
\newcommand{\Integ}{\mathbb Z}

\newcommand{\comments}[1]{}
\begin{document}

\author[Wes Camp and Michael Mihalik]{Wes Camp and Michael Mihalik}
\address{Department of Mathematics\\
        Vanderbilt University\\
        Nashville, TN 37240}
\email{w.camp@vanderbilt.edu,  michael.l.mihalik@vanderbilt.edu}

\title[]{A classification of right-angled Coxeter groups with no 3-flats and locally connected boundary}%
%\author{Wes Camp and Michael Mihalik}%
\maketitle
%
% ----------------------------------------------------------------

% ----------------------------------------------------------------
\begin{abstract}
If $(W,S)$ is a right-angled Coxeter system and $W$ has no $\mathbb Z^3$ subgroups, then it is shown that the absence of an elementary separation property in the presentation diagram for $(W,S)$ implies all CAT(0) spaces acted on geometrically by $W$ have locally connected CAT(0) boundary. It was previously known that if the presentation diagram of a general right-angled Coxeter system satisfied the separation property then all CAT(0) spaces acted on geometrically by $W$ have non-locally connected boundary. In particular, this gives a complete classification of the right-angled Coxeter groups with no 3-flats and with locally connected boundary.
\end{abstract}

\section{Introduction}
In this paper, we classify the right-angled Coxeter groups with no $\mathbb Z^3$ subgroups that have locally connected CAT(0) boundary. We say a CAT(0) group has locally (respectively, non-locally) connected boundary if all CAT(0) boundaries of the group are locally (respectively non-locally) connected.  Our main theorem states that if the Coxeter presentation of the group satisfies an elementary combinatorial condition then this group has locally connected boundary and otherwise has non-locally connected boundary. This condition was first considered in \cite{MRT}, where in fact, the authors conjectured that any right-angled Coxeter group has locally connected boundary iff the group presentation satisfies this condition. The primary working tool for both this paper and \cite{MRT} is the notion of a {\it filter} for  CAT(0) geodesics $r$ and $s$ in a CAT(0) space $X$ on which the Coxeter group $W$ acts geometrically. A filter is a connected, one-ended planar graph whose edges are labeled by the Coxeter generators $S$ of $W$. Hence there is a natural (proper) map of the filter into the Cayley graph of $(W,S)$, which in turn maps properly and $W$-equivariantly into the CAT(0) space $X$. The two sides of the filter track the geodesics $r$ and $s$ and the limit set of the filter is a connected set in $\partial X$ (the boundary of $X$), containing the limit points of $r$ and $s$. The idea is to construct a filter in such a way so that if $r$ and $s$ are ``close" in $\partial X$, then the filter has ``small" limit set containing the limit points $r$ and $s$, and local connectivity of the boundary of $X$ follows. 

In \cite{MRT}, two combinatorial conditions are defined on the Coxeter presentation of the group. The main theorem there states: if the first  condition holds then the Coxeter group has non-locally connected boundary, and if neither condition holds then the Coxeter group has locally connected boundary. In fact, when neither condition holds, the filters constructed in \cite{MRT} basically have hyperbolic geometry; i.e. any geodesic path in the Cayley graph from the base point of the filter to another point of the filter must track the filter geodesic connecting these two points (just as in a word hyperbolic group). In this paper, the geometry of our filters is necessarily more complex. The no $\mathbb Z^3$ subgroup hypothesis does restrict the pathology of the geometry of the filter,  but our results are the natural next step towards a full classification of right-angled Coxeter groups with locally connected boundary, and give hard evidence that the conjectured classification of \cite{MRT} is sound.

If a Coxeter group has no $\mathbb Z^2$ subgroup then it is word hyperbolic \cite{Mss}, and all one-ended word hyperbolic groups have (unique) locally connected boundary \cite{Swa}. Januszkiewicz and Swiatkowski (\cite{JS}) produce word hyperbolic, right-angled Coxeter groups of virtual cohomological dimension $n$ for all positive integers $n$, so our no $\mathbb Z^3$ hypothesis does not restrict the virtual cohomological dimension of the groups under consideration. In \cite{CK}, Croke and Kleiner exhibit a one-ended CAT(0) group with non-homeomorphic boundaries. Each of these boundaries is in fact connected but not path connected (see \cite{CMT}). In particular, (by classical point set topology)  these boundaries are not locally connected. It seems that many of the serious pathologies one sees in boundaries of CAT(0) groups, but not in boundaries of word hyperbolic groups, happen in the presence of non-local connectivity. At the time of this writing, no CAT(0) group has been shown to have non-homeomorphic boundaries, one of which is locally connected. There are numerous questions about how or even if the homology and homotopy of two boundaries of a CAT(0) group can differ. These questions may be more tractable if the boundaries considered are locally connected. If our results extend to all right-angled Coxeter groups then those with locally connected boundary should provide an interesting testing ground for such questions.

The paper is laid out as follows. In section 2,  basic definitions and background results are listed, including  a lemma (\ref{diamond}) that provides the fundamental combinatorial tool for measuring how large the limit set of a filter might be. In section 3, the basics of CAT(0) spaces and groups are outlined,
and we list two tracking results (developed in \cite{MRT}) that connect the CAT(0) geometry and algebraic combinatorics  of right-angled Coxeter groups. In section 4, we develop the idea of a filter, define a virtual factor separator and prove our main theorem:

\medskip

\begin{thm} \label{MTh}
Suppose $(W,S)$ is a one-ended right-angled Coxeter system  that has no visual subgroup isomorphic to $(\mathbb Z_2 \ast \mathbb Z_2)^3$. 
\begin{enumerate}
\item If $W$ visually splits as $(\mathbb Z_2\ast\mathbb Z_2)\times A$  then $A$ is word hyperbolic, $W$ has unique boundary homeomorphic to the suspension of  the boundary of $A$, and the boundary of $W$ is non-locally connected iff $A$ is infinite ended.
\item Otherwise, $W$ has locally connected boundary iff $(W,S)$ has no virtual factor 
separator.
\end{enumerate}
 \end{thm}

\begin{cor}
Suppose $(W,S)$ is a one-ended right-angled Coxeter system  that has no visual subgroup isomorphic to $(\mathbb Z_2 \ast \mathbb Z_2)^3$. Then 
all CAT(0) boundaries of $W$ are locally connected or all are non-locally connected.
\end{cor}

The group $W$ visually splits as in item (1) of the theorem precisely when there are $s,t\in S$ such that $m(s,t)=\infty$ and $\{s,t\}$ commutes with $S-\{s,t\}$, so this condition is easily checked. 
If a CAT(0) group splits as $G=(\mathbb Z_2\ast \mathbb Z_2)\times A$ then any boundary of $G$ is the suspension of a boundary of $A$ (see \cite{MRT}) and this suspension is locally connected iff the boundary of $A$ is locally connected. If $(W,S)$ is a one-ended right-angled Coxeter system with no visual subgroup isomorphic to $(\mathbb Z_2 \ast \mathbb Z_2)^3$ and $W$ visually splits as $(\mathbb Z_2 \ast \mathbb Z_2)\times A$ , then $A$  is word hyperbolic (see \cite {Mss}). It is straightforward to check if a Coxeter group is infinite-ended (see Remark \ref{infend}). Thus item (1) of the theorem is easily verified and the real content of the theorem is contained in item (2). If $C\subset S$ is a virtual factor separator for $(W,S)$, then $W$ visually splits as an amalgamated product $\langle A\rangle\ast_{\langle C\rangle} \langle B\rangle$ (here $A$ and $B$ are subsets of $S$ with $A\cup B=S$ and $A\cap B=C$). Therefore local connectivity of boundaries of $W$ is directly tied to visual splittings.

In section 5, we give examples to show there are no combination or splitting results for right-angled Coxeter groups that respect local connectivity of boundaries. One example describes a right-angled Coxeter group as the (visual) amalgamated product $W=A\ast_CB$ where $A$ and $B$ are one-ended and word hyperbolic (so both have locally connected boundary) and $C$ is virtually a surface group (with boundary a circle), but $W$ has non-locally connected boundary. The second example describes a right-angled Coxeter group $W$ that visually splits as $A\ast_CB$, and a single element of infinite order in $C$ determines a boundary point of non-local connectivity in both $A$ and $B$. Nevertheless, our main theorem implies $W$ has locally connected boundary. 
These examples indicate there are no reasonable graph of groups approaches to this problem. Morse theory also seems unhelpful, but we do not expand here.

\section{Preliminaries}

We use \cite{Bou} and  \cite{Davis} as basic references for the results in this section.
\begin{defn}
A \textbf{ Coxeter system} is a pair $(W,S)$, where $W$ is a group with \textbf{Coxeter presentation}: 

\[
\langle{}S:(st)^{m(s,t)}\rangle
\]

\noindent where $m(s,t)\in\{1,2,\ldots, \infty\}$, $m(s,t)=1$ iff $s=t$, and $m(s,t)=m(t,s)$. The relation $m(s,s)=1$ means each generator is of order 2, and  $m(s,t)=2$, iff $s$ and $t$ commute.
\end{defn}

\begin{defn}
We call a Coxeter group $(W,S)$ \textbf{right-angled} if $m(s,t)\in\{2,\infty\}$ for all $s\neq t$.
\end{defn}

\bigskip

We are only interested in right-angled Coxeter groups in this paper but we state many of the lemmas of this section in  full generality.
In what follows, we will let $\Lambda=\Lambda(W,S)$ denote an abbreviated version of the Cayley graph for $W$ with respect to the generating set $S$. As usual, the vertices of $\Lambda$ are the elements of $W$, and there is an edge between the vertices $w$ and $ws$ for each $s\in S$, but instead of having two edges between adjacent vertices in the graph (since each generator has order 2), we allow only one. 

\begin{defn}
For a Coxeter system $(W,S)$, the \textbf{presentation graph} $\Gamma(W,S)$ for $(W,S)$ is the graph with vertex set $S$ and an edge labeled $m(s,t)$ connecting distinct $s,t\in{}S$ when $m(s,t)\ne \infty$.
\end{defn}

\begin{defn}\label{bar}
For a Coxeter system $(W,S)$, a {\bf word} in $S$ is an $n$-tuple $w=[a_1,a_2,\ldots, a_n]$, with each $a_i\in S$. Let $\overline w\equiv a_1\cdots a_n\in W$. We say the word $w$ is {\bf geodesic}  if there is no word $[b_1,b_2,\ldots, b_m]$ such that $m<n$ and $\overline w=b_1\cdots b_m$. Define $lett(w)\equiv\{a_1,\ldots , a_n\}$.
\end{defn} 

\begin{defn}\label{rearr}
For a  Coxeter system $(W,S)$, let $\overline e\in S$ be the label of the edge $e$ of $\Lambda(W,S)$. An {\bf edge path}  $\alpha\equiv (e_1,e_2,\ldots, e_n)$ in a graph $\Gamma$ is a map $\alpha:[0,n]\to \Gamma$ such that $\alpha$ maps $[i,i+1]$ isometrically to the edge $e_i$. For $\alpha$  an edge path in $\Lambda(W,S)$, let $lett(\alpha)\equiv \{\overline e_1, \ldots , \overline e_n\}$, and let $\overline \alpha\equiv\overline e_1\cdots \overline e_n$. If $\beta$ is another geodesic with the same initial and terminal points as $\alpha$, then call $\beta$ a {\bf rearrangement} of $\alpha$.
\end{defn}

\begin{lemma}
\label{sameletters}
Suppose $(W,S)$ is a Coxeter system, and $a$ and $b$ are $S$-geodesics for $w\in{}W$ (so $w=\overline a=\overline b$). Then $lett(a)=lett(b)$.
\end{lemma}

\begin{defn}
If $(W,S)$ is a Coxeter system and $A\subset S$, then $lk(A)\equiv\{t\in{}S: m(a,t)=2$ for all $a\in A\}$. So when $(W,S)$ is right-angled, $lk(A)$ is the combinatorial link of $A$ in $\Gamma(W,S)$, and the subgroups $\langle A\rangle$ and $\langle lk(A)\rangle$ of $W$ commute.
\end{defn}

\begin{lemma} \textbf{(The Deletion Condition)}.
Suppose $(W,S)$ is a Coxeter system. If the $S$-word $w=[a_{1},a_{2},\ldots,a_{n}]$  is not geodesic, then two of the $a_i$ delete; i.e. we have for some $i<j$, $\overline w=a_{1}a_{2}\cdots{}a_{n}=a_{1}a_{2}\cdots{}a_{i-1}a_{i+1}\cdots{}a_{j-1}a_{j+1}\cdots{}a_{n}$.
\end{lemma}

For a  Coxeter system $(W,S)$, an edge path $\alpha=(e_1,e_2, \ldots, e_n)$ in $\Lambda (W,S)$ is geodesic iff the word $[\overline e_1,\overline e_2,\ldots, \overline e_n]$ is geodesic. If $\alpha$ is not geodesic and $\overline e_i$ deletes with $\overline e_j$, for $i<j$, let $\tau$ be the  the path beginning at the end point of $e_{i-1}$ with edge labels $[\overline e_{i+1},\ldots ,\overline e_{j-1}]$. Then $\tau$ ends at the initial point of $e_{j+1}$, so that $(e_1,\ldots, e_{i-1},\tau, e_{j+1},\ldots, e_n)$ is a path with the same end points as $\alpha$. We say the edges $e_i$  and $e_j$ {\bf delete} in $\alpha$. 

\begin{defn}
If $(W,S)$ is a Coxeter system and $A\subset{}S$, then the subgroup of $W$ generated by $A$ is called a \textbf{special} (or \textbf{visual}) subgroup of $W$.
\end{defn}

\begin{lemma}
Suppose $(W,S)$ is a Coxeter system, and $A\subset{}S$. Then the special subgroup $\langle{}A\rangle$ of $W$ has Coxeter (sub)-presentation

\[
\langle{}A:(st)^{m(s,t)};s,t\in{}A\rangle
\]

In particular, distinct $s,t\in{}S$ determine unique elements of $W$, and $m(s,t)$ is the order of $st$ for all $s,t\in{}S$.
\end{lemma}

\begin{lemma}
\label{differentletters}
Suppose $(W,S)$ is a Coxeter system, and $U,V\subset{}S$, with $U\cap{}V=\emptyset$. If $u$ is a geodesic in the letters of $U$ and $v$ is a geodesic in the letters of $V$, then $[u,v]$ is an $S$-geodesic.
\end{lemma}

\begin{defn}
For $(W,S)$ a Coxeter system and $\alpha$ a geodesic in $\Lambda(W,S)$, let $B(\alpha)\equiv\{\overline e\in{}S:$ $e$ is a $\Lambda$-edge at the end of $\alpha$ and $(\alpha,e) \text{ is not geodesic}\}$.
\end{defn}

\begin{lemma}
\label{finiteback}
Suppose $(W,S)$ is a Coxeter system, and $\alpha$ a geodesic in $\Lambda$. Then $B(\alpha)$ generates a finite group.
\end{lemma}

%\begin{lemma}
%Suppose $(W,S)$ is a Coxeter system, $U\subset{}S$. If $s\in{}S-U$, $t\in{}S$, and $sut\in{}\langle{}U\rangle$ for some $u\in{}\langle{}U\rangle$, then
%\begin{enumerate}
%\item $t=s$,
%\item $sus=u$ and,
%\item if $u=u_1{}u_2{}\cdots{}u_n$ is of minimal length ($u_i\in{}U$), then $su_i{}s=u_i$ for $1\leq{}i\leq{}n$.
%\end{enumerate}
%\end{lemma}

%\begin{lemma}
%Suppose $(W,S)$ is a Coxeter system. Then $s,t\in{}S$ are conjugate in $W$ iff there is an odd labeled path in $\Gamma(W,S)$ connecting $s$ and $t$. In particular, if $s$ and $t$ delete in some $S$-word, then they are equal in $S$ or connected by an odd labeled path.
%\end{lemma}

\begin{lemma}
\label{sameletter}
If $(W,S)$ is a right-angled Coxeter system, and $s,t\in S$ delete in some $S$-word. Then $s=t$.
\end{lemma}

\begin{lemma}
\label{radel}
Suppose $(W,S)$ is a right-angled Coxeter system, $[a_1, a_2, \dots, a_n]$ is geodesic and $[a_1, a_2, \dots, a_n, a_{n+1}]$ is not. Then $a_{n+1}$ deletes with some $a_m$. If $i\neq n+1$ is the largest integer such that ${a_i}={a_{n+1}}$, then $a_{n+1}$ deletes with $a_i$ and ${a_{n+1}}$ commutes with each letter ${a_{i+1}},{a_{i+2}},\dots,{a_n}$.
\end{lemma}

\begin{defn}
Suppose  $\Gamma$ is the presentation graph of a Coxeter system $(W,S)$, and $C\subset S$ separates the vertices of $\Gamma$. Let $A'$ be the vertices of a component of $\Gamma-C$ and $B=S-A'$. Let $A=A'\cup C$. Then $W$ splits as $\langle A\rangle \ast_{\langle C\rangle}\langle B\rangle$ (see \cite{MT}) and this splitting is called a {\bf visual splitting} for  $(W,S)$.
\end{defn}

\begin{defn}
Let $(W,S)$ be a Coxeter system, and let $e$ be an edge of $\Lambda(W,S)$ with initial vertex $v\in W$. The \textbf{wall} $w(e)$ is the set of edges of $\Lambda(W,S)$ each fixed (setwise) by the action of the conjugate $v\overline ev^{-1}$ on $\Lambda$.
\end{defn}

\begin{remark}\label{wallsw}
Certainly $e\in w(e)$ and if $d$ is an edge of  $w(e)$,  with vertices $u$ and $w$, then $(v\overline ev^{-1})u=w$  and $(v\overline ev^{-1})w=u$. Also,  $\Lambda(W,S)-w(e)$ has exactly two  components and these components are interchanged by the action of $v\overline ev^{-1}$ on $\Lambda(W,S)$.

If $(W,S)$ is right-angled, then given an edge $a$ of $\Lambda(W,S)$ with initial vertex $y_1$ and terminal vertex $y_2$, $a$ is in the same wall as $e$ iff there is an edge path $(t_1, \dots, t_n)$ in $\Lambda(W,S)$ based at $w_1$ so that $w_1 \overline t_1 \cdots \overline t_n=y_1$ and $w_2 \overline t_1 \cdots \overline t_n=y_2$, where $y_1$ and $y_2$ are the vertices of $e$ and $m(\overline e, \overline t_i)=2$ for each $1\leq i \leq n$.
\end{remark}

\begin{defn}
Let $(W,S)$ be a right-angled Coxeter system. We say the walls $w(e)\neq w(d)$ of $\Lambda(W,S)$ {\bf cross} if there is a relation square in $\Lambda(W,S)$ with edges in $w(e)$ and $w(d)$. 
\end{defn}

\begin{remark} \label{wallprops}
We have the following basic properties of walls in a right-angled Coxeter system $(W,S)$:
\begin{enumerate}
\item If edges $a$ and $e$ of $\Lambda(W,S)$ are in the same wall, then $\overline a=\overline e $. 

\item Being in the same wall is an equivalence relation on the set of edges of $\Lambda(W,S)$.

\item If $(e_1,e_2,\ldots,e_n)$ is an edge path in $\Lambda(W,S)$ then $e_i$ and $e_j$ are in the same wall iff $\overline e_i$ and $\overline e_j$ delete in the word $[e_1,e_2,\ldots, e_n]$. Furthermore, the path $(e_{i+1}',\ldots, e_{j-1}')$ that begins at the initial point of $e_i$, and has the same labeling as $(e_{i+1},\ldots ,e_{j-1})$, ends at the end point of $e_j$ and $w(e_k)=w(e_k')$ for all $i<k<j$. If $\gamma$ is a path in $\Lambda(W,S)$, then $\gamma$ is geodesic iff no two edges of $\gamma$ are in the same wall.

\item If $\gamma$ and $\tau$ are geodesics in $\Lambda(W,S)$ between the same two points, then the edges of $\gamma$ and $\tau$ define the same set of walls.

\end{enumerate}
\end{remark}

The basics of van Kampen diagrams can be found in Chapter 5 of \cite{LS}. Suppose $(W,S)$ is a right-angled Coxeter system. We need only consider relation squares with boundary labels $abab$ in van Kampen diagrams for right-angled Coxeter groups (since those of the type $aa$ are easily removed).
Let $(w_1, \dots, w_n)$ be an edge path loop in $\Lambda(W,S)$, so $\overline w_1\dots \overline w_n=1$ in $W$.  Consider a van Kampen diagram $D$ for this word. For a given boundary edge $d$ of $D$ (corresponding to say $w_i$), $d$ can belong to at most one relation square of $D$ and there is an edge  $d_1$ opposite $d$ on this square. Similarly, if $d_1$ is not a boundary edge, it belongs to a unique relation square adjacent to the one containing $d$ and $d_1$. Let $d_2$ be the edge opposite $d_1$ in the second relation square.  These relation squares define a {\it band} in $D$ starting at $d$ and ending at say $d'$ on the boundary of $D$ and corresponding to some $w_j$ with $j\neq i$. This means that $w_i$ and $w_j$ are in the same wall. However, $w_k$ and $w_\ell$ being in the same wall does not necessarily mean that they are part of the same band in $D$; but if $(w_1, \dots, w_r)$ and $(w_{r+1}, \dots, w_n)$ are both geodesic, then by (3) in the above remark, bands in $D$ correspond exactly to walls in $\Lambda(W,S)$. This is the situation we will usually consider.

\begin{lemma}
\label{shortback}
Let $(W,S)$ be a right-angled Coxeter system, and let $\gamma$ be a geodesic in $\Lambda(W,S)$ with initial vertex $x$ and terminal vertex $y$. Let $A$ be a set of edges of $\gamma$, and $\tau_A$ be a shortest path based at $x$ containing an edge in the same wall as $a$ for all $a\in A$. Then $\tau_A$ can be extended to a geodesic to $y$.
\end{lemma}

\begin{proof}
%Note that the last edge of $\tau_a$ is in the same wall as $a$ and $\tau_a$ is geodesic. 
Let $v$ denote the endpoint of $\tau_A$, and let $\lambda$ be a geodesic from $v$ to $y$.  Let $\tau_A=(a_1, \dots, a_n)$ and consider a van Kampen diagram $D$ for $(\tau_A, \lambda, \gamma^{-1})$. If $W(a_j)=W(a)$ for some $a\in A$ and the band for $a_j$ does not end on $\gamma$, then it must end on $\lambda$, by (3) of Remark \ref{wallprops}. However, then the band for $a$ cannot end on $\lambda$, $\gamma$, or $\tau_A$ (which is impossible). 
Therefore the band for $a_j$ must end on the edge of $D$ corresponding to the edge $a$ of $\gamma$. Now suppose for some $1\leq i \leq n$, the band for $a_i$ ends on $\lambda$. Deleting edges of $(\tau_A, \lambda)$ corresponding to this shared wall gives a path shorter than $\tau_A$ with an edge in the same wall as $a$ for all $a\in A$ (see Remark \ref{wallprops} (3)), a contradiction. 

%\vspace{2mm}
%\begin{center}
%\includegraphics{wallpf.eps}
%\end{center}
%\begin{figure}
%\caption{}
%\end{figure}

Therefore, all bands on $\lambda$ and $\tau_a$ end on $\gamma$, so $(\tau_a, \lambda)$ has the same length as $\gamma$ and is therefore  geodesic.
\end{proof}

The following lemma has some of its underlying ideas in Lemma 5.10 of \cite{MRT}. It is an important tool for measuring the size of (connected) sets in the boundaries of our groups and is used repeatedly in our proof of the main theorem.

\begin{lemma}
\label{diamond}
Suppose $(W,S)$ is a right-angled Coxeter system, and $(\alpha_1,\alpha_2)$ and $(\beta_1,\beta_2)$ are geodesics in $\Gamma(W,S)$ between the same two points. There exist geodesics $(\gamma_1,\tau_1), (\gamma_1,\delta_1), (\delta_2,\gamma_2)$, and $(\tau_2,\gamma_2)$ with the same end points as $\alpha_1,\beta_1,\alpha_2,\beta_2$ respectively, such that:
\begin{enumerate}
\item $\tau_1$ and $\tau_2$ have the same edge labeling,
\item $\delta_1$ and $\delta_2$ have the same edge labeling, and
\item $lett(\tau_1)$ and $lett(\delta_1)$ are disjoint and commute.
\end{enumerate}
Furthermore, the paths $(\tau_1^{-1},\delta_1)$ and $(\delta_2,\tau_2^{-1})$ are geodesic. 
\end{lemma}

\begin{center}
\includegraphics{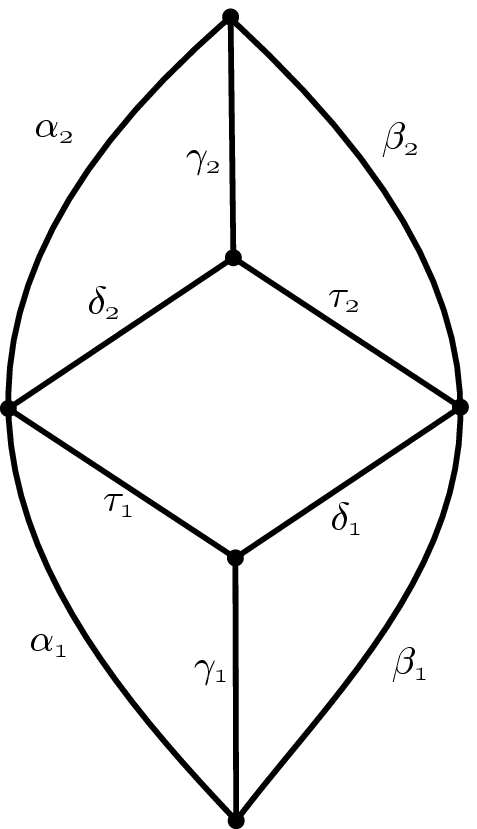}
\end{center}

\centerline{\textsc{Figure 1}}

\medskip

\begin{proof}
%Certainly the lemma is true if $\abs{\alpha_1}=\abs{\alpha_2}= 1$.  Assume the lemma is true if the length of $(\alpha_1,\alpha_2)$ is $k-1$ for some $k\geq 3$. Now assume the length of $(\alpha_1,\alpha_2)$ is $k$ and $\abs{\alpha_2}\geq\abs{\beta_2}$. Let $y$ be the end point of $\alpha_1$ and suppose $\beta_2=(x_1,\dots,x_n)$. Consider the non-geodesic path $(\alpha_1,\alpha_2,x_n)$. If $x_n$ deletes with $x_n'$ in $\alpha_1$, then by Lemma \ref{radel}, $\overline x_n \not\in lett(\alpha_2)$ and $\overline x_n$ commutes with  $lett(\alpha_2)$. Also, by lemma \ref{radel}, we may assume that $x_n' $ is the last edge of $\alpha_1$. Let $\alpha_1'$ be the initial segment of $\alpha_1$ sans $x_n'$, $\alpha_2'$ is the path at the end point of $\alpha_1'$ with the same labeling as $\alpha_2$  and $\beta_2' $ is the initial segment of $\beta_2$ sans $x_n$. Apply our induction hypothesis to the bigon for $(\alpha_1', \alpha_2')$ and $(\beta_1,\beta_2')$ and obtain the paths $\gamma_i'$, $\tau_i' $ and $\delta_i'$ satisfying the conclusion of the lemma and as pictured in figure ???????(b). 
%To conclude this case, let $\tau_1=(\tau_1',x_n')$, $\tau_2=(\tau_2',x'')$ where $\overline x_n=\overline x_n'=\overline x_n''$. Let $\delta_1=\delta_1'$ and $\delta_2$ be the path at $y$ with the same labeling as $\delta_2'$ (equivalently $\delta_1'$). Let $\gamma_1=\gamma_1'$, and $\gamma_2$ have the same labeling as $\gamma_2'$.  
%If $x_n$ deletes in $\alpha_2$ with the edge $x_n'$ then proceed as in figure ??????(c). 

Consider a van Kampen diagram for the loop $(\alpha_1, \alpha_2, \beta_2^{-1}, \beta_1^{-1})$ (Figure 1), and recall that since $(\alpha_1, \alpha_2)$ and $(\beta_1, \beta_2)$ are geodesic, bands in this van Kampen diagram correspond exactly to walls in $\Lambda(W,S)$. Let $a_1, \dots, a_n$ be the edges of $\alpha_1$ (in the order they appear on $\alpha_1$) that are in the same wall as an edge of $\beta_1$. Notice that if $e$ is an edge of $\alpha_1$ occurring before $a_1$, then $w(e)$ crosses $w(a_1)$. Therefore $\alpha_1$ can be rearranged to begin with an edge in $w(a_1)$, since $\overline{a}_1$ commutes with every edge label of $\alpha_1$ before it. Similarly, $w(a_2)$ must cross $w(e)$ for any edge $e\neq a_1$ of $\alpha_1$ occurring before $a_2$, so $\alpha_1$ can be rearranged to begin with an edge in $w(a_1)$ followed by an edge in $w(a_2)$. Continuing for each $a_i$ gives us a rearrangement $(\gamma_1, \tau_1)$ of $\alpha_1$ where the walls of $\gamma_1$ are exactly $w(a_1), \dots, w(a_n)$. If $b_1, \dots, b_m$ are the edges of $\beta_1$ in the same wall as an edge of $\alpha_1$, then the same process gives us a rearrangement $(\gamma_1', \delta_1)$ of $\beta_1$ where the walls of $\gamma_1'$ are exactly $w(b_1), \dots, w(b_m)$. However, $\{w(a_1), \dots, w(a_n)\}=\{w(b_1), \dots, w(b_m)\}$, so $m=n$ and $\gamma_1$ and $\gamma_1'$ are geodesics between the same points, so $(\gamma_1, \delta_1)$ is a rearrangement of $\beta_1$. Construct rearrangements $(\delta_2, \gamma_2)$ and $(\tau_2,\gamma_2)$ of $\alpha_2$ and $\beta_2$ respectively in the same way, and note that $\tau_1$ and $\tau_2$ have the same walls, $\delta_1$ and $\delta_2$ have the same walls, and every wall of $\tau_1$ crosses every wall of $\delta_1$. In particular, (see Remark \ref{wallprops} (3)) ($\tau_1^{-1},\delta_1)$ is geodesic.
\end{proof}

\begin{remark}
For the entirety of this paper, we will only consider the case of Lemma \ref{diamond} where $\abs{\alpha_1}=\abs{\beta_1}$. In this case, $\abs{\tau_1}=\abs{\tau_2}=\abs{\delta_1}=\abs{\delta_2}$, so the \textbf{diamond} formed by the loop $\tau_1^{-1} \delta_1 \tau_2 \delta_2^{-1}$ is actually a product square. If $y$ is the endpoint of $\alpha_1$ and $\mu$ is any other geodesic between the same points as $(\alpha_1, \alpha_2)$, the diamond between $(\alpha_1, \alpha_2)$ and $\mu$ at $y$ is therefore well defined. We call $\tau_1^{-1}$ the \textbf{down edge path} at $y$ and $\delta_2$ the {\bf up edge path} at $y$ of the diamond for $(\alpha_1, \alpha_2)$ and $(\beta_1, \beta_2)$.
\end{remark}

\begin{lemma}
\label{doublediamond}

Suppose $(W,S)$ is a right-angled Coxeter system with no visual subgroup isomorphic to $(\Integ_2 * \Integ_2)^3$. Let $\lambda_1$, $\lambda_2$, $\lambda_3$ be $\Lambda(W,S)$-geodesics between two points $a$ and $b$, and let $x_1$, $x_2$, $x_3$ be points on $\lambda_1$, $\lambda_2$, $\lambda_3$ respectively, such that the $x_i$ are all equidistant from $a$. Let $\nu_{12}$ and $\nu_{13}$ be the down edge paths respectively of the diamonds at $x_1$ between $\lambda_1$ and $\lambda_2$ and between $\lambda_1$ and $\lambda_3$, as in Lemma \ref{diamond}, and suppose $|\nu_{12}|\geq |\nu_{13}|\geq 2{\abs S}$. If $\{a,b\}\subset lett(\nu_{12})\cap lett(\nu_{13})$ and $m(a,b)=\infty$, then $d(x_2, x_3) < 2(|\nu_{12}|-|\nu_{13}|) + 4{\abs S}$. 

\end{lemma}

\begin{center}
\includegraphics{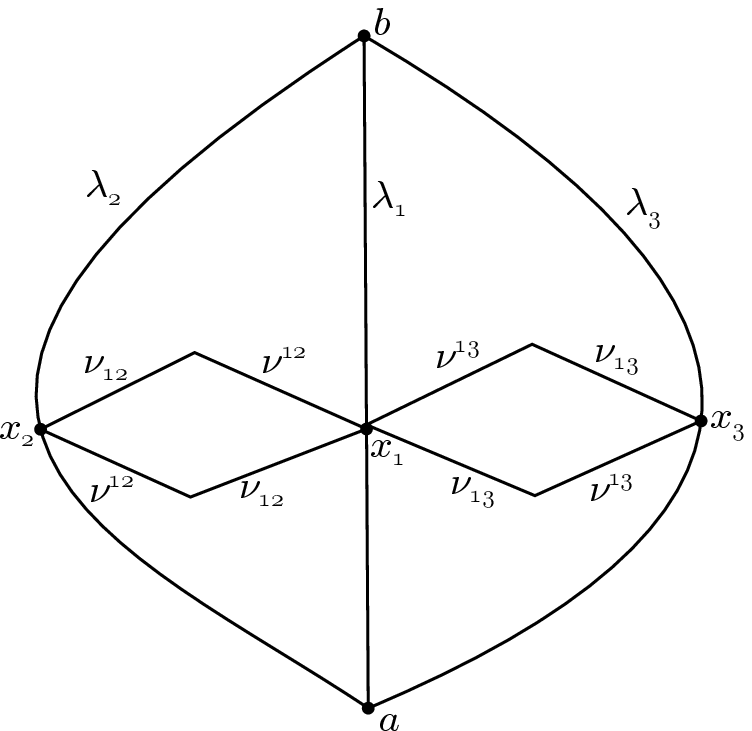}
\end{center}
\centerline{\textsc{Figure 2}}

\medskip

\begin{proof} To simplify notation we use the same label for two  paths with the same edge labeling.  
Let $\nu^{12}$ and $\nu^{13}$ be the up edge paths respectively of the diamonds at $x_1$ between $\lambda_1$ and $\lambda_2$ and between $\lambda_1$ and $\lambda_3$. Note that at $x_2$, $\nu^{12} \nu_{12} \nu_{13} \nu^{13}$ is a path from $x_2$ to $x_3$. By Lemma \ref{diamond},  $\{a,b\}$ is disjoint from and commutes with $lett(\nu^{12})\cup lett(\nu^{13})$. Thus, $\nu^{13}$ cannot have a pair of walls with unrelated labels cross a pair of walls with unrelated labels from $\nu^{12}$, since that would give a visual $(\Integ_2 * \Integ_2)^3$ in $W$. Rearrange  $\nu^{12}$ and $\nu^{13}$ so they have a longest common initial segment (see definition \ref{rearr}). As $\nu^{12}$ and $\nu^{13}$ are initial segments of a geodesic from $x_1$ to $b$, the walls of the unshared edges of $\nu^{13}$ cross those of $\nu^{12}$. In particular, the unshared part of $\nu^{13}$ has length $\leq {\abs S}-1$, and $\nu^{12}$ and $\nu^{13}$ share two walls with unrelated labels. By symmetry, this last part implies $\nu_{13}$ and $\nu_{12}$ at $x_1$ can be rearranged to have a shared initial segment so the unshared part of $\nu_{13}$ has length $\leq {\abs S}-1$. Deleting edges of the path $\nu^{12} \nu_{12} \nu_{13} \nu^{13}$ (from $x_2$ to $x_3$) corresponding to these shared walls leaves us with a geodesic from $x_2$ to $x_3$ of length less than $2(|\nu_{12}|-|\nu_{13}|) + 4{\abs S}$.
\end{proof}

\section{CAT(0) Spaces and Actions by Coxeter Groups}

\begin{defn}
A metric space $(X,d)$ is \textbf{proper} if each closed ball is compact.
\end{defn}

\begin{defn}
Let $(X,d)$ be a complete proper metric space. Given a geodesic triangle $\triangle abc$ in $X$, we consider a \textbf {comparison triangle} $\triangle \overline a\overline b \overline c$ in $\Real^2$ with the same side lengths. We say $X$ satisfies the \textbf{CAT(0) inequality} (and is thus a \textbf{CAT(0) space}) if, given any two points $p,q$ on a triangle $\triangle abc$ in $X$ and two corresponding points $\overline p,\overline q$ on a corresponding comparison triangle $\triangle \overline a\overline b\overline c$, we have 
\begin{center}
$d(p,q)\leq d(\overline p,\overline q)$.
\end{center}
 
\end{defn}

\begin{prop}
If $(X,d)$ is a CAT(0) space, then
\begin{enumerate}
\item the distance function $d:X\times X \rightarrow\Real$ is convex,
\item $X$ has unique geodesic segments between points, and
\item $X$ is contractible.
\end{enumerate}
\end{prop}

\begin{defn}
A \textbf{geodesic ray} in a CAT(0) space $X$ is an isometry $[0, \infty)\rightarrow X$.
\end{defn}

\begin{defn}
Let $(X,d)$ be a proper CAT(0) space. Two geodesic rays $c,c':[0,\infty)\rightarrow X$ are called \textbf{asymptotic} if for some constant $K$, $d(c(t),c'(t))\leq K$ for all $t\in[0,\infty)$. Clearly this is an equivalence relation on all geodesic rays in $X$, regardless of basepoint. We define the \textbf{boundary} of $X$ (denoted $\partial X$) to be the set of equivalence classes of geodesic rays in $X$. We denote the union $X\cup\partial X$ by $\overline X$.
\end{defn}

\noindent The next proposition guarantees that the topology we wish to put on the boundary is independent of our choice of basepoint in $X$.

\begin{prop}
Let $(X,d)$ be a proper CAT(0) space, and let $c:[0,\infty)\rightarrow X$ be a geodesic ray. For a given point $x\in X$, there is a unique geodesic ray based at $x$ which is asymptotic to $c$.
\end{prop}

\noindent For a proof of this (and more details on what follows), see \cite{BH}.

\bigskip

We wish to define a topology on $\overline X$ that induces the metric topology on $X$. Given a point in $\partial X$, we define a neighborhood basis for the point as follows: 

\noindent Pick a basepoint $x_0\in X$. Let $c$ be a geodesic ray starting at $x_0$, and let $\epsilon>0$, $r>0$. Let $S(x_0,r)$ denote the sphere of radius $r$ based at $x_0$, and let $p_r:X\rightarrow S(x_0,r)$ denote the projection onto $S(x_0,r)$. Define
\begin{center}
$U(c,r,\epsilon)=\{x\in\overline X :d(x,x_{0})>r,d(p_{r}(x),c(r))<\epsilon\}.$
\end{center}
This consists of all points in $\overline X $ whose projection onto $S(x_0,r)$ is within $\epsilon$ of the point of the sphere through which $c$ passes. These sets together with the metric balls in $X$ form a basis for the \textbf{cone topology}. The set $\partial X$ with this topology is sometimes called the \textbf{visual boundary}. For our purposes, we will just call it the boundary of $X$.

\begin{defn}
We say a finitely generated group $G$ \textbf{acts geometrically} on a proper geodesic metric space $X$ if there is an action of $G$ on $X$ such that:
\begin{enumerate}
\item Each element of $G$ acts by isometries on $X$,
\item The action of $G$ on $X$ is cocompact, and
\item The action is properly discontinuous. 
\end{enumerate}
\end{defn}

\begin{defn} 
We call a group $G$ a \textbf{CAT(0) group} if it acts geometrically on a CAT(0) space.
\end{defn}

The next theorem, due to Milnor \cite{Mil}, will be used in conjunction with the next two technical lemmas to identify geodesic rays in $X$ with certain rays in a right-angled Coxeter group which acts on $X$.

\begin{thm}
If a group $G$ with a finite generating set $S$ acts geometrically on a proper geodesic metric space $X$, then $G$ with the word metric with respect to $S$ is quasi-isometric to $X$ under the map $g\mapsto g\cdot x_0$, where $x_0$ is a fixed base point in $X$.
\end{thm}

Let $(W,S)$ be a right-angled Coxeter group acting geometrically on a CAT(0) space $X$. Pick a base point $*\in X$ and identify a copy of the Cayley graph for $(W,S)$ inside $X$ as in the previous theorem. If vertices $u,v$ of $\Lambda(W,S)$ are adjacent, then we connect $u*$ and $v*$ with a CAT(0) geodesic in $X$. This defines a map $C:\Lambda \rightarrow X$ respecting the action of $W$. If $\alpha$ is a $\Lambda$-geodesic, we call $C(\alpha)$ a $\Lambda$-geodesic in $X$.

\begin{defn}
Let $r:[a,b]\rightarrow X$ be a geodesic segment in $X$ with $r(a)=x$ and $r(b)=y$. For $\delta>0$, we say that a Cayley graph geodesic $\alpha$ \textbf{$\delta$-tracks} $r$ if every point of $C(\alpha)$ is within $\delta$ of a point of the image of $r$ and the endpoints of $r$ and $C(\alpha)$ are within $\delta$ of each other.
\end{defn}

Proofs of the next two lemmas can be found in section 4 of \cite{MRT}.

\begin{lemma}
\label{deltatrack}
There exists some $\delta_1>0$ such that for any geodesic ray $r:[0,\infty)\rightarrow X$ based at $x_0$, there is a geodesic ray $\alpha_r$ in $\Lambda(W,S)$ that $\delta_1$-tracks $r$.
\end{lemma}

\begin{lemma}
\label{sameseg}
There exist $c,d>0$ such that, for any infinite geodesic rays $r$ and $s$ and $X$ based at $x_0$ that are within $\epsilon$ of each other at distance $M$ from $x_0$, there are Cayley graph geodesic rays $\alpha$ and $\beta$ which ($c\epsilon+d$)-track $r$ and $s$ respectively, and which share a common initial segment of length $M-c\epsilon-d$.
\end{lemma}

\section{Local connectivity and filter construction}

\begin{defn} 
We say a CAT(0) group $G$ has \textbf{(non-)locally connected boundary} if for every CAT(0) space $X$ on which $G$ acts geometrically, $\partial X$ is (non-)locally connected.
\end{defn}

\begin{defn}
Let $(W,S)$ be a right-angled Coxeter system, and let $\Gamma$ be the presentation graph for $(W,S)$. A \textbf{virtual factor separator} for $(W,S)$ (or $\Gamma$) is a pair $(C,D)$ where $D\subset C\subset S$, $C$ separates vertices of $\Gamma$, $\langle C-D \rangle$ is finite and commutes with $\langle D \rangle$, and there exist $s,t\in S-D$  such that $m(s,t)=\infty$ and  $\{s,t\}$ commutes with $D$.
\end{defn}

In this section we prove the main theorem (\ref{MTh}). Part (1) of this result is clear. If the right angled Coxeter system $(W,S)$ does not visually split as a direct product $(\mathbb Z_2\ast\mathbb Z_2)\times A$ and has a virtual factor separator, then $W$ has non-locally connected boundary (see \cite{MRT}).
It remains to show local connectivity of the boundaries of CAT(0) spaces acted upon geometrically by one-ended right-angled Coxeter groups with no virtual factor separators. To do this, we pick two rays whose end points are ``close" in $\partial X$, and use Lemma \ref{sameseg} to find two tracking Cayley geodesics which share a long initial segment. We then construct a filter of geodesics (a way of ``filling in" the space) between the branches of these Cayley geodesics such that its limit set gives a small connected set in $\partial X$ containing our original rays.

In what follows, let $(W,S)$ be a right-angled, one-ended Coxeter system with no virtual factor separator and containing no visual subgroup isomorphic to $(\mathbb Z_2 \ast \mathbb Z_2)^3$. Set $N=\abs{S}$. 
We will show that if $W$ acts geometrically on a CAT(0) space $X$, then given $\epsilon > 0$, there exists $\delta$ such that if two points $x,y \in \partial X$ satisfy $d(x,y) < \delta$, then there is a connected set in $\partial X$ of diameter $\leq \epsilon$ containing $x$ and $y$.

\begin{rem}\label{infend}
The right-angled Coxeter group $W$ is one-ended iff $\Gamma(W,S)$ contains no complete separating subgraph (i.e., a subgraph whose vertices generate a finite group in $W$). For a proof of this, see \cite{MT}.
\end{rem}

\begin{rem} 
If $e$ is an edge in $\Lambda(W,S)$, we let $\overline e\in S$ denote the label of $e$. Recall that for $g\in W$, $B(g)$ is the set of $s\in S$ such that $gs$ is shorter than $g$, and that $\langle B(g) \rangle$ is finite (Lemma \ref{finiteback}). 
\end{rem}

\begin{rem}
If $\alpha$ is a geodesic in $\Lambda(W,S)$ from a vertex $a$ to another vertex $b$, then for any other geodesic $\gamma$ from $a$ to $b$, we have $B(\alpha)=B(\gamma)$. Since this set depends only on $a$ and $b$, we may use the notation $B(b\rightarrow a)$ to denote $B(\alpha)$, where it is more convenient to do so.
\end{rem}

We begin with an example that demonstrates one important idea behind our proof. Let $(W,S)$ be a right-angled Coxeter system where $W$ is one-ended and acts geometrically on a CAT(0) space $X$. Suppose that $(e_1, e_2, \dots, e_m, e_{m+1}, e_{m+2}, \dots)$ and $(e_1, e_2, \dots, e_m, d_{m+1}, d_{m+2}, \dots)$ are $\Lambda$-geodesics in $X$, based at a vertex $*$, that $(c + d)$-track two CAT(0) geodesics $r$ and $s$ in $X$ (as in Lemma \ref{sameseg}), and let $x_m$ denote the endpoint of $(e_1, \dots, e_m)$. Set $a_1=\overline e_{m+1}$ and $b_1=\overline d_{m+1}$. By the previous remarks, $B(x_m\rightarrow *)$ does not separate the presentation graph $\Gamma(W,S)$, and $a_1,b_1\notin B(x_m\rightarrow *)$. Let $a_1, t_1, \dots, t_k, b_1$ be the vertices of a path from $a_1$ to $b_1$ in $\Gamma(W,S)$ where each $t_i \notin B(x_m \rightarrow *)$. We can construct the following (labeled) planar diagram (Figure 3) that maps naturally into $\Lambda$ (respecting labels) and then to $X$:

\vspace{1mm}

\begin{center}
\includegraphics{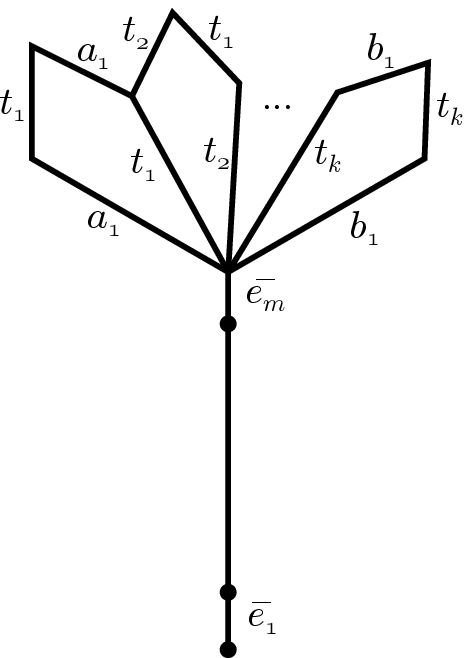}
\end{center}
\centerline{\textsc{Figure 3}}

\vspace{3mm}

As in \cite{MRT}, we call Figure 3 a \textbf{fan} for the geodesics $(e_1, \dots, e_m, e_{m+1})$ and $(e_1, \dots, e_m, d_{m+1})$. Each loop corresponds to the relation given by $t_i$ and $t_{i+1}$ commuting. Since each $t_i$ commutes with $t_{i+1}$ and $t_i, t_{i+1}\notin B(x_m \rightarrow *)$, the path $(e_1, \dots, e_m, t_i, t_{i+1})$ is geodesic for each $i$ (this is an easy consequence of Lemma \ref{radel}). Now, let $a_2=\overline e_{m+2}$, $b_2=\overline d_{m+2}$, and continue. We overlap our original fan with fans for the pairs of geodesics $(e_1, \dots, e_m, e_{m+1}, e_{m+2})$ and $(e_1, \dots, e_m, e_{m+1}, t_1)$, $(e_1, \dots, e_m, t_1, a_1)$ and $(e_1, \dots, e_m, t_1, t_2)$, and so on, ending with a fan for $(e_1, \dots, e_m, d_{m+1}, t_k)$ and $(e_1, \dots, e_m, d_{m+1}, d_{m+2})$. 

By continuing to build fans in this manner, we construct (Figure 4) a connected, one-ended, planar graph (with edge labels in $S$) called a \textbf{filter} for the geodesics $(e_1, e_2, \dots, e_m, e_{m+1}, e_{m+2}, \dots)$ and $(e_1, e_2, \dots, e_m, d_{m+1}, d_{m+2}, \dots)$. Note that if $v$ is a vertex of the filter, then the obvious edge paths in the filter from $*$ to $v$ define $\Lambda$-geodesics. The limit set determined by this filter in $\partial X$ is a connected set containing our original rays $r$ and $s$. However, this connected set may not be small. We refer to the image of a filter, in $\Lambda$ or in $X$, again as a filter. 
\vspace{1mm}

\begin{center}
\includegraphics{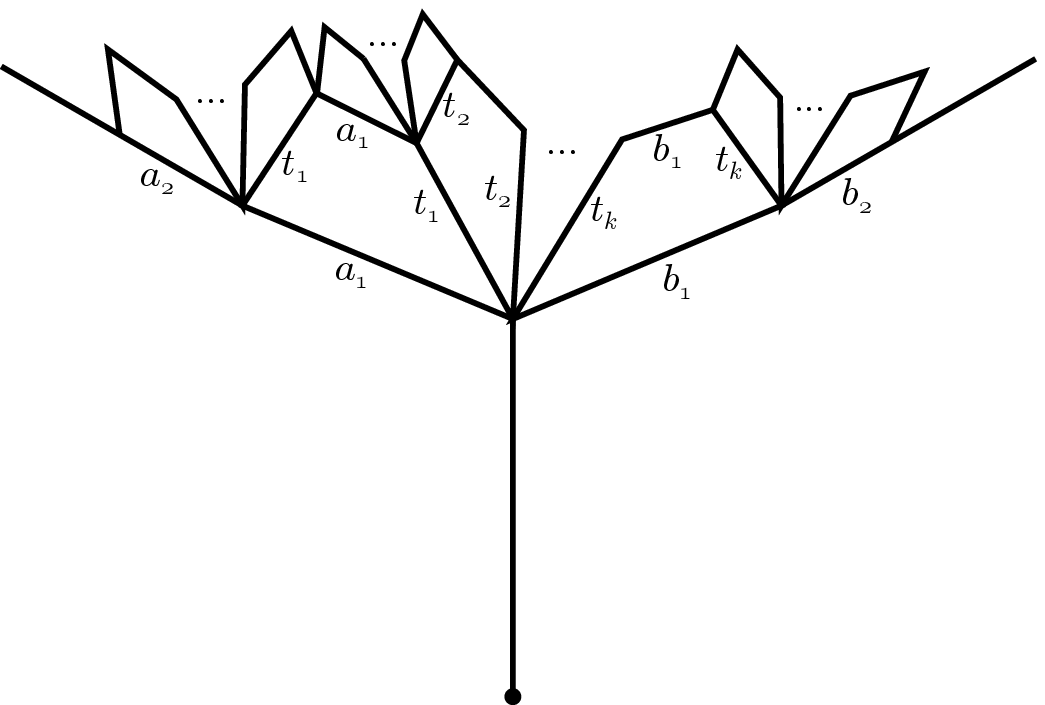}
\end{center}

\centerline{\textsc{Figure 4}}

\vspace{3mm}

If we wish for the limit set of our filter to be small in $\partial X$, we need to ensure that the CAT(0) geodesics between $*$ and points in our filter are not far from the base point $x_m$ of our filter. Using Lemma \ref{diamond}, we know what a wide bigon between two geodesics in $\Lambda$ must look like. Our first goal is to classify the ``down edge paths'', from  $x_m$ towards $\ast$, of any potential diamond given by a wide bigon in $\Lambda$, and show there are only two ``types" of such paths. 

\begin{rem}
For the rest of this section, we assume that $\Gamma$ has no virtual factor separators and $(W,S)$ contains no visual subgroup isomorphic to $(\Integ_2 \ast \Integ_2)^3$.
\end{rem}

\begin{defn}
Construct a geodesic from $x_m$ to $*$ in $\Lambda$ as follows: let $\alpha_1$ be a longest geodesic with edge labels in the finite group $\langle B(x_m \rightarrow *) \rangle$, and let $y_1$ be the endpoint of $\alpha_1$ based at $x_m$. Let $\alpha_2$ be a longest geodesic in the finite group $\langle B(y_1 \rightarrow *) \rangle$. Continuing in this way, we obtain a geodesic $(\alpha_1, \alpha_2, \dots, \alpha_r )$ from $x_m$ to $*$. We call this a \textbf{back combing} geodesic from $x_m$ to $*$. 
\end{defn}

\begin{remark}
\label{backprops}
We have the following properties of a back combing geodesic $(\alpha_1, \alpha_2, \dots, \alpha_r )$ from $x_m$ to $*$:

\begin{enumerate}
\item Every edge label of $\alpha_i$ commutes with every other edge label of $\alpha_i$.

\item No edge label of $\alpha_{i+1}$ commutes with every edge label of $\alpha_i$. 

\item Let $(\gamma_1, \gamma_2)$ be a $\Lambda$-geodesic from $x_m$ to $*$ and let $v$ be the endpoint of $\gamma_1$. If $(\beta_1, \beta_2, \dots, \beta_s)$ is a back combing geodesic from $x_m$ to $v$, then the set of walls of $\beta_i$ is a subset of the set of walls of $\alpha_i$, for $1\leq i \leq s$. In particular: 

\item Let $(\gamma_1, \gamma_2)$ be a $\Lambda$-geodesic from $x_m$ to $*$. If $\gamma_1$ has an edge in the same wall as an edge of $\alpha_j$ for some $1 \leq j \leq r$, then $\gamma_1$ contains an edge in the same wall as an edge of $\alpha_i$ for all $1\leq i \leq j$.

\item Let $(\gamma_1, \gamma_2)$ and $(\tau_1, \tau_2)$ be $\Lambda$-geodesics from $x_m$ to $*$. If each of $\tau_1$, $\gamma_1$, and $\alpha_j$ (for some $1\leq j \leq r$) has an edge of the wall $w(e)$, then for each $1\leq i \leq j$, each of $\alpha_i$, $\tau_1$, and $\gamma_1$ has an edge of the wall $w(e_i)$.

\end{enumerate}
\end{remark}

We will always assume that $x_m$ and $*$ are sufficiently far apart, so for now suppose $d(x_m, *) > 7N^2$. Let $\alpha_{7N+1}= (u_1, u_2, \dots, u_k)$ (note $k < N$), and for $1\leq i \leq k$, let $U_i$ be a shortest $\Lambda$-geodesic based at $x_m$ such that  last edge of $U_i$ is in the same wall as 
$u_i$ (so by Lemma \ref{shortback}, $U_i$ extends to a geodesic from $x_m$ to $*$). There may be several such geodesics, but they all have the same set of walls.

\begin{lemma}
\label{arrangeu}
If $(\gamma_1, \gamma_2)$ is a $\Lambda$-geodesic from $x_m$ to $*$ with $\abs{\gamma_1} \geq 7N^2$, then $\gamma_1$ can be rearranged to begin with exactly $U_i$, for some $1\leq i \leq k$.
\end{lemma}

\begin{center}
\includegraphics{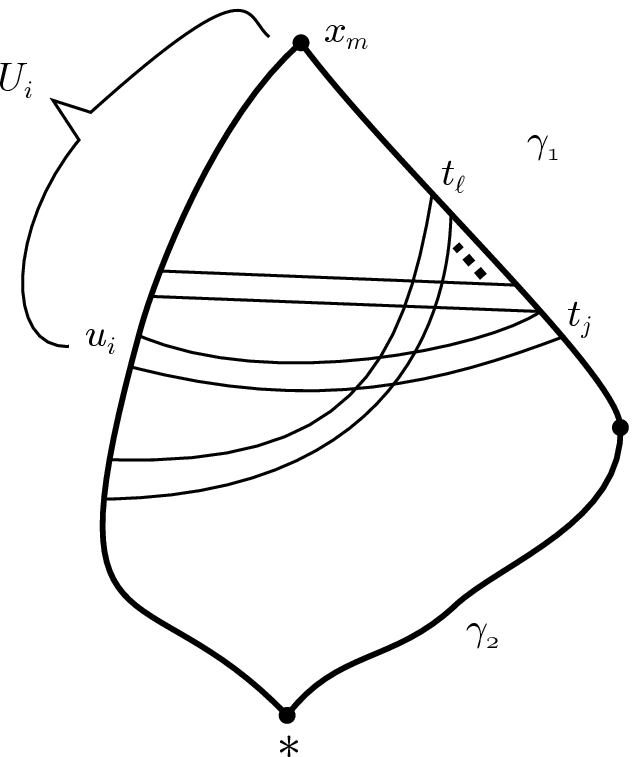}
\end{center}

\centerline{\textsc{Figure 5}}

\vspace{5mm}

\begin{proof}

Consider a van Kampen diagram  (Figure 5) for the geodesic bigon determined by $(\gamma_1, \gamma_2)$ and a $\Lambda$-geodesic from $x_m$ to $*$ that begins with $U_i$. Let $\gamma_1=(t_1, t_2, \dots, t_s)$, where $s\geq 7N^2$. Let $j$ be the smallest number such that the edge $t_j$ shares a wall with an edge $u_i$ of $\alpha_{7N+1}$, for some $1\leq i \leq k$ (such a $j$ exists from Remark \ref{backprops} (3) and because the lengths of $\alpha_1,\dots,\alpha_{7N}$ are each less than $N$). Now, choose $\ell$ maximal with $1 < \ell < j$ where the wall of $t_\ell$ is not on $U_i$. Clearly the wall of $t_\ell$ crosses the walls of $t_{\ell+1}, \dots, t_j$, so $\overline t_\ell$ commutes with $\overline t_{\ell+1}, \dots, \overline t_j$, and so $\overline t_1\cdots \overline t_j$ can be rewritten $\overline t_1  \cdots \overline t_{\ell-1} \overline t_{\ell+1} \cdots \overline t_j  \overline t_\ell $. Repeating this process, we obtain a rearrangement of $\gamma_1$ that begins with a rearrangement of $U_i$, which can be replaced by $U_i$.
\end{proof}

We now have a finite number $k < N$ of ``directions'', given by our $U_i$, in which a bigon can be wide at $x_m$. The next lemma (\ref{twodir}) is a fundamental combinatorial consequence of our no $(\Integ_2 * \Integ_2)^3$ hypothesis which allows us to refine this collection to at most two directions.

We will say that $U_i$ and $U_j$ \textbf{R-overlap} if there is an edge $a$ of $\alpha_{R}$ that shares a wall with an edge of $U_i$ and an edge of $U_j$. Let $\tau_a$ be a shortest $\Lambda$-geodesic based at $x_m$ that can be extended to a geodesic ending at $*$ and whose last edge is in the same wall as $a$. Then $U_i$ and $U_j$ can be rearranged to begin with $\tau_a$. We will now refine our list of $U_i$ through the following five step process which at each application either terminates the process, or removes at least one of the $U_i$ from our list and replaces all those that remain by geodesics with last edge in a wall of $\alpha_R$ where $R$ begins at $7N$ and is reduced by one at each successive application:

\begin{enumerate}
\item Choose $i$ minimal so that for some $j > i$, $U_i$ and $U_j$ $R$-overlap  (by sharing some wall with an edge $a$ of $\alpha_{R}$). If no such $i$ exists, our process stops.
\item Replace $U_i$ with a shortest geodesic based at $x_m$ and ending with an edge in the wall of $a$ (which can be extended to a geodesic to $*$ by Lemma \ref{shortback}), and redefine $u_i$ to be $a$.
\item Eliminate $U_j$ from the list of $U_\ell$.
\item For each remaining $U_\ell$ with $\ell \ne i$, choose an edge of $U_\ell$ in the same wall as an edge $b_\ell$ of $\alpha_{R}$, replace $U_\ell$ with a shortest geodesic based at $x_m$ and ending with an edge in the wall of $b_\ell$, and redefine $u_\ell$ to be $b_\ell$.
\item At this point each $U_\ell$ ends with an edge sharing a wall with an edge of $\alpha_{R}$. If two $U_\ell$ end with edges in the same wall, remove one of them from the list. Now, relabel the remaining $U_\ell$ to form a list $U_1, \dots, U_p$. Reduce $R$ to $R-1$.
\end{enumerate}

When this process stops, no two $U_i$ $R$-overlap, and each $u_i$ shares a wall with an edge of $\alpha_{R+1}$. Since $U_i$ is a shortest geodesic with last edge in the wall of $u_i$, every geodesic from $x_m$ to the end point of $U_i$ ends with $u_i$. By the minimality of $U_i$ and  Remark \ref {backprops} (3),  if $c$ is an edge of $U_i$ in a  wall of $\alpha_R$, then $\overline u_i$ and $\overline c$ do not commute. Note that when this process stops, $6N < R \leq 7N$.

\begin{lemma}\label{twodir}
At most two $U_i$ survive this reduction process.
\end{lemma}

\begin{proof}
Suppose none of $U_1$, $U_2$, and $U_3$ $R$-overlap. Let $a_1$, $a_2$, $a_3$ be edges of $U_1$, $U_2$, $U_3$ respectively such that each $a_i$ shares a wall with an edge of $\alpha_{R}$. Since the process terminated, $\overline a_1 , \overline a_2, \overline a_3$ are distinct and commute. But $a_i$ does not commute with $u_i$ for $i=1,2,3$, and the pairs $(a_i, u_i)$ all commute, so this gives a visual $(\Integ_2 \ast \Integ_2)^3$ in $(W,S)$, a contradiction.
\end{proof}

We now have at most two directions $U_1$ and $U_2$ remaining. If there is no $U_2$, then to simplify notation for now, define $U_2=U_1$.
 
If there is no geodesic extension of $\beta=(e_1, \dots, e_m)$ that can be rearranged to form a bigon of width $16N^2$ with the down edge path of the diamond at $x_m$ (Lemma \ref{diamond}) containing every wall of $U_2$, then we redefine $U_2=U_1$, and similarly for $U_1$. If no geodesic extension of $\beta$ can lead to a wide bigon in either direction, then an arbitrary filter (built as in the example in the beginning of this section) has ``small" connected set limit set in $\partial X$.
 
Note that $U_1$ and $U_2$ have length at least $6N$. Now, if $U_1$ and $U_2$ share two walls with unrelated labels, then let $(\alpha_1, \alpha_2, \dots)$ be a back combing from $x_m$ to the endpoint of $U_1$, and choose an edge $a$ in $\alpha_2$ so that $U_1$ and $U_2$ both have edges in the same wall as $a$ (such an edge exists by (5) of Remark \ref{backprops}). Let $U_1=U_2$ be a shortest geodesic at $x_m$  containing an edge in the same wall as $a$. 

\begin{remark}
\label{ucommute}
If $U_1\neq U_2$, then $U_1$ and $U_2$ share less than $N$ walls, and the sets $lett(U_1) - (lett(U_1)\cap lett(U_2))$ and $lett(U_2) - (lett(U_1)\cap lett(U_2))$ commute.
\end{remark}

For this next remark, note that $x_m$ is the $(m+1)^{st}$ vertex of $\beta$ (since $*$ is the first).

\begin{remark}
\label{notboth}
If $U_1\neq U_2$, $(\beta, \gamma)$ is a $\Lambda$-geodesic and $\gamma'$ is some rearrangement of $(\beta, \gamma)$ whose $(m+1)^{st}$ vertex is of distance at least $14N^2$ from $x_m$, then the down edge path $\tau$ at $x_m$ of the diamond (Lemma \ref{diamond}) for these two geodesics can be rearranged to begin with either $U_1$ or $U_2$, by Lemma \ref{arrangeu}. Both cannot initiate rearrangements of $\tau$, since otherwise there is a $(\Integ_2 * \Integ_2)^2$ in $lett(\tau)$, and the diamond at $x_m$ containing $\tau$ determines a $(\Integ_2 * \Integ_2)^3$ in $(W,S)$.
\end{remark}

Recall that $(e_1, e_2, \dots, e_m, e_{m+1}, e_{m+2}, \dots)$ and $(e_1, e_2, \dots, e_m, d_{m+1}, d_{m+2}, \dots)$ are geodesics in $\Lambda$ $(c + d)$-tracking two CAT(0) geodesics in $X$, and $x_m$ is the endpoint of $(e_1, \dots, e_m)$. Let $x_i$ denote the endpoint of $(e_1, \dots, e_i)$ where $i>m$, and $y_i$ denote the endpoint of $d_i$ where $i>m$. Set $U_1^{x_m}=U_1$ and construct $U_1^{x_i}$, $U_2^{x_i}$, $U_1^{y_i}$, $U_2^{y_i}$ exactly as above, by replacing $x_m$ with $x_i$ or $y_i$.  

Let $\lambda=(\ell_1, \ell_2, \dots, \ell_n)$ be a geodesic based at some $x_i$ extending $(\beta, e_{m+1}, \dots, e_i)$ (or based at $y_i$ and extending $(\beta, d_{m+1}, \dots, d_i)$), but not passing through $e_{i+1}$ ($d_{i+1}$). Our goal is to classify the directions back toward $*$ at the endpoint of $\lambda$ in a way that gives us some correspondence between our direction(s) at $x_i$ ($y_i$) and the direction(s) at the endpoint of $\lambda$. We'll do this inductively, by corresponding directions at the endpoint of each edge of $\lambda$ to the directions at the endpoint of the previous edge of $\lambda$. For what follows, let $v$ denote the endpoint of $\ell_1$. 

\begin{enumerate}

\item If $U_1^{x_i}=U_2^{x_i}$ and $\overline \ell_1$ commutes with $lett(U_1)$, then let $U_1^{x_i}(\ell_1)=U_2^{x_i}(\ell_1)$ be the edge path at $v$ with the same labeling as $U_1^{x_i}$. Note that if $\overline \ell_1 $ commutes with $lett(U_1)$, then $\overline \ell_1 \notin lett(U_1)$, since $(\ell_1^{-1},U_1) $ is geodesic.

\item If $U_1^{x_i}=U_2^{x_i}$ and $\overline \ell_1 $ does not commute with $lett(U_1^{x_i})$, then set $U_1^{x_i}((\ell_1))=U_2^{x_i}((\ell_1))=({\ell_1}^{-1},U_1^{x_i})$.

\item If $U_1^{x_i}\neq U_2^{x_i}$, we construct directions from $v$ back toward $*$ just as we've done from $x_m$ back toward $*$. If there is only one direction $V_1$, set $U_1^{x_i}(\ell_1)=U_2^{x_i}(\ell_1)=V_1$. If there are two directions $V_1$ and $V_2$, but there is no geodesic extension of $(\beta, e_{m+1}, \dots, e_i, \ell_1)$ that can lead to a $16N^2$ wide bigon in the $V_2$ direction at $v$, then set $U_1^{x_i}(\ell_1)=U_2^{x_i}(\ell_1)=V_1$ (and similarly for $V_1$). If there is no geodesic extension that can lead to a wide bigon in either direction, then building arbitrary fans, as in the example at the beginning of this section, fills in this section of the filter with rays in $X$ that are sufficiently close to our original two rays in $X$. Otherwise, take a geodesic extension $\gamma$ of $(\beta, e_{m+1}, \dots, e_i, \ell_1)$ so that a rearrangement of $(\beta, e_{m+1}, \dots, e_i, \ell_1, \gamma)$ gives a $16N^2$ wide bigon at $v$ whose down edge path of the diamond at $v$ (Lemma \ref{diamond}) begins with $V_1$. By Remark \ref{notboth}, the down edge path of the diamond at $x_i$ for this bigon can be rearranged to begin with either $U_1^{x_i}$ or $U_2^{x_i}$ (but not both). If it's $U_1^{x_i}$ set $U_1^{x_i}(\ell_1)=V_1$ and $U_2^{x_i}(\ell_1)=V_2$, else set $U_1^{x_i}(\ell_1)=V_2$ and $U_2^{x_i}(\ell_1)=V_1$. It will be made clear by Lemma \ref{sharedwalls} that this choice does not depend on the choice of $\gamma$.

\end{enumerate}

We now define $U_1^{x_i}((\ell_1, \ell_2))$ and $U_2^{x_i}((\ell_1, \ell_2))$ by replacing $U_1^{x_i}$ by $U_1^{x_i}(\ell_1)$ and $U_2^{x_i}$ by $U_2^{x_i}(\ell_1)$ in the above process, and continue repeating this process to define $U_1^{x_i}(\lambda)$ and $U_2^{x_i}(\lambda)$. Note that for any geodesic extension $(\lambda_1, \lambda_2)$ of $(\beta, e_{m+1}, \dots, e_i)$ that does not pass through $e_{i+1}$, if $U_1^{x_i}(\lambda_1)=U_2^{x_i}(\lambda_1)$, then $U_1^{x_i}((\lambda_1, \lambda_2))=U_2^{x_i}((\lambda_1, \lambda_2))$. 

\begin{remark}
From here on, when we mention a geodesic extension $\lambda$ of $(\beta, e_{m+1}, \dots, e_i)$ (or $(\beta, d_{m+1}, \dots, d_i))$, we assume $\lambda$ does not pass through $e_{i+1}$ ($d_{i+1}$).
\end{remark}

\begin{lemma}
\label{arrangeu2}
Let $\lambda$ be a geodesic extension of $(\beta, e_{m+1}, \dots, e_i)$ (or $(\beta, d_{m+1}, \dots, d_i))$ with $U_1^{x_i}(\lambda)\neq U_2^{x_i}(\lambda)$ ($U_1^{y_i}(\lambda)\neq U_2^{y_i}(\lambda)$), and let $(\gamma_1, \gamma_2)$ be any geodesic from the endpoint of $\lambda$ to $*$. If $\abs{\gamma_1} \geq 7N^2$, then $(\gamma_1, \gamma_2)$ can be rearranged to begin with either $U_1^{x_i}(\lambda)$ or $U_2^{x_i}(\lambda)$.
\end{lemma}

\begin{proof} 
This follows from the proof of Lemma \ref{arrangeu} and the construction of the $U_i^{x_i}(\lambda)$.
\end{proof}

\begin{remark}
Remarks \ref{ucommute} and \ref{notboth} apply to $U_1^{x_i}(\lambda)$ and $U_2^{x_i}(\lambda)$, whenever they are not equal.
\end{remark}

\begin{lemma}
\label{sharedwalls}
Suppose $\lambda$ geodesically extends $(\beta, e_{m+1}, \dots, e_i)$, $e$ is an edge with $(\beta, e_{m+1}, \dots, e_i, \lambda, e)$ geodesic, and $U_1^{x_i}((\lambda, e))\neq U_2^{x_i}((\lambda, e))$, then $U_j^{x_i}((\lambda, e))$ and $U_j^{x_i}(\lambda)$ have at least $6N-3$ walls in common.
\end{lemma}

\begin{proof}

It suffices to show this for $U_1$ $(=U_1^{x_m})$ and $U_1(\ell_1)$ $(=U_1^{x_m}(\ell_1))$, as in the first step of our $U_i(\lambda)$ construction. Let $\gamma$ be the geodesic extension of $(\beta, \ell_1)$ used in the construction of the $U_i(\ell_1)$, so that there is a rearrangement $\gamma'$ of $(\beta, \ell_1, \gamma)$ whose $(m+2)^{nd}$ vertex is at least $16N^2$ from the endpoint of $(\beta, \ell_1)$. Let $\tau$ be the down edge path at the endpoint of $\ell_1$ for the diamond for these two geodesics, as in Lemma \ref{diamond}. Note $\abs{\tau} \geq 8N^2$. By Lemma \ref{arrangeu2} (and without loss of generality), $\tau$ can be rearranged to begin with $U_1(\ell_1)$. However, if $\tau$ has an edge in the same wall as $\ell_1$ then $\tau$ can be rearranged to begin with $\ell_1$, and so $(\ell_1, U_1)$. Otherwise, $\tau$ can be rearranged to begin with $U_1$, so either way every edge of $U_1$ shares a wall with an edge of $\tau$. Let $(\alpha_1, \dots, \alpha_{6N}, \dots)$ be a back combing from $x_m$ to $*$, choose an edge $a_1$ of $\alpha_{6N-1}$ that shares a wall with an edge of $U_1(\ell_1)$, and pick an edge $a_2$ of $\alpha_{6N-2}$ whose label does not commute with $\overline a_1$ (so $a_2$ also shares a wall with an edge of $U_1(\ell_1)$). Pick an edge $b_1$ of $\alpha_{6N-2}$ that shares a wall with an edge of $U_1$, and pick an edge $b_2$ of $\alpha_{6N-3}$ whose label doesn't commute with $\overline b_1$. If neither $b_1$ nor $b_2$ have their walls on $U_1(\ell_1)$, then the pair $\overline a_1, \overline a_2 $ commutes with the pair $\overline b_1, \overline b_2$, and the up edge path at $x_m$ for this diamond gives a third pair of unrelated elements that commute with the pairs $\overline a_1$, $\overline a_2$ and $\overline b_1$, $\overline b_2$, which is a contradiction. Thus the wall of $b_2$ must cross $U_1(\ell_1)$, and so $U_1(\ell_1)$ and $U_1$ have at least $6N-3$ walls in common.
\end{proof}

We claimed in the construction of the $U_j^{x_i}(\lambda)$ that Lemma \ref{sharedwalls} shows the association between $U_j^{x_i}$ and $U_j^{x_i}(\ell_1)$ is independent of the choice of $\gamma$. If the association depended on the choice of $\gamma$, then by the above proof, $U_1^{x_i}(\ell_1)$ would have $6N-3$ walls in common with both $U_1^{x_i}$ and $U_2^{x_i}$. By Remark \ref{ucommute}, $lett(U_1^{x_i}(\ell_1))$ must then contain a $(\Integ_2 * \Integ_2)^2$, meaning the walls of $U_1^{x_i}(\ell_1)$ cannot all appear on the down edge path at $x_m$ of the diamond for a wide bigon, which is a contradiction.

This next lemma gives an important correspondence between the directions $U_j^{x_i}(\lambda_1)$ and $U_j^{x_i}((\lambda_1, \lambda_2))$.

\begin{lemma}
\label{nochange}
Let $(\lambda_1, \lambda_2, \lambda_3)$ be a geodesic extending $(\beta, e_{m+1}, \dots, e_i)$ (not passing through $x_{i+1}$) with endpoint $v$, let $\tau$ be another $\Lambda$-geodesic from $*$ and $v$, let $z_J$ and $z_M$ denote the endpoints of $\lambda_1$ and $\lambda_2$, respectively, and suppose $U_1^{x_i}((\lambda_1, \lambda_2))\neq U_2^{x_i}((\lambda_1, \lambda_2))$. Suppose $R\geq 14N^2$ and every vertex $z_J, z_{J+1}, \dots, z_M$ of $\lambda_2$ is of $\Lambda$-distance at least $R$ from $\tau$. If the down edge path of the diamond at $z_J$ for $\tau$ and $(\beta, e_{m+1}, \dots, e_i, \lambda_1, \lambda_2, \lambda_3)$ can be rearranged to begin with $U_1^{x_i}(\lambda_1)$, then the down edge path of the diamond at $z_M$ for these geodesics can be rearranged to begin with $U_1^{x_i}((\lambda_1, \lambda_2))$ (and similarly for $U_2$).
\end{lemma}

\begin{proof}
It suffices to show this for $U_1^{x_m}((\lambda_1, \lambda_2))=U_1((\lambda_1, \lambda_2))$ when $(\lambda_1, \lambda_2, \lambda_3)$ is a geodesic based at $x_m$, since the constructions are identical for each $x_i$. Let $\gamma_J$ and $\gamma_M$ be the down edge paths at $z_J$ and $z_M$ respectively of the diamonds for $(\beta, \lambda_1, \lambda_2, \lambda_3)$ and $\tau$, as given by Lemma \ref{diamond}. (See Figure 6.)

\vspace{2mm}

\begin{center}
\includegraphics{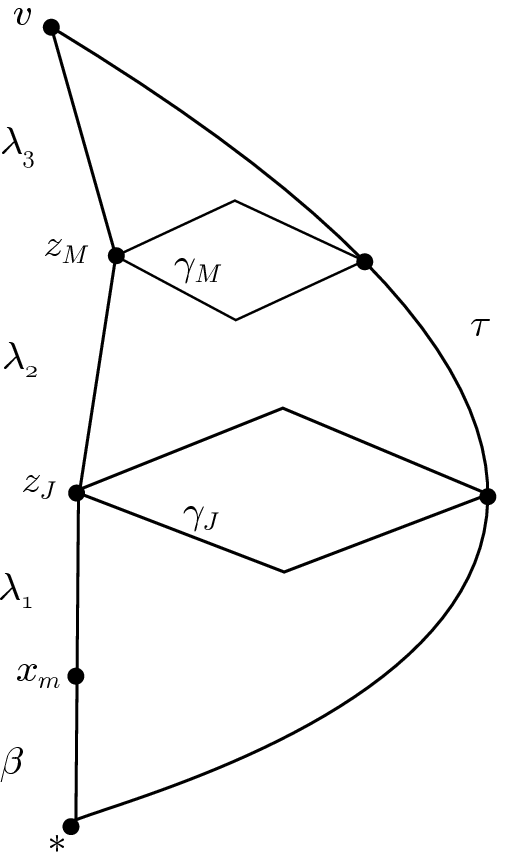}
\end{center}
\centerline{\textsc{Figure 6}}

\vspace{2mm}
For each $K$ with $J < K < M$, let $\lambda_K$ denote the initial segment of $(\lambda_1, \lambda_2)$ ending at $z_K$. Suppose $\gamma_J$ can be rearranged to begin with $U_1(\lambda_1)$ but $\gamma_M$ cannot be arranged to begin with $U_1((\lambda_1,\lambda_2))$. 
There is then $K$ with $J < K < M$ where the down edge path $\gamma_K$ at $z_K$ of the diamond for these geodesics can be rearranged to begin with $U_1(\lambda_K)$ and the down edge path $\gamma_{K+1}$ at $z_{K+1}$ can be rearranged to begin with $U_2(\lambda_{K+1})$, by Lemma \ref{arrangeu2}. By Lemma \ref{sharedwalls} and since $U_1(\lambda_{K+1})\neq U_2(\lambda_{K+1})$, there is a pair of unrelated edge labels $a_1,b_1$ of $U_1(\lambda_K)$ that commute with some unrelated pair of labels $a_2,b_2$ from $U_2(\lambda_{K+1})$. Let $\nu^K$ and $\nu^{K+1}$ be the up edge paths of the diamonds at $z_K$ and $z_{K+1}$ respectively. From Lemma \ref{diamond}, these paths differ by at most two walls, and so they have two unrelated edge labels $a_3$ and $b_3$ in common. But then the pairs $(a_i,b_i)$ must all commute, giving a visual $(\Integ_2 * \Integ_2)^3$ in $W$, a contradiction. 
\end{proof}

The proof of the next lemma basically follows that of Lemma 5.5 of \cite{MRT}.

\begin{lemma}
\label{avoidlink}
Let $\lambda$ be a geodesic based at $x_i$ extending $(\beta, e_{m+1}, \dots, e_i)$ with endpoint $v$, and let $s$ and $t$ be vertices of $\Gamma$ not in $B(v\rightarrow *)$. If $(\gamma_1,\gamma_2)$ is any rearrangement of $(\beta, e_{m+1}, \dots, e_i, \lambda)$ where $\langle lett(\gamma_2) \rangle$ is infinite, then there is a path from $s$ to $t$ of length at least two in $\Gamma$, none of whose vertices (except possibly $s$ and $t$) are in $\text{lk}(lett(\gamma_2))\cup B(v\rightarrow *)$.
\end{lemma}

\begin{proof}
Since $(\beta, e_{m+1}, \dots, e_i, \lambda)$ can be rearranged to end with $\gamma_2$, for $e\in B(v\rightarrow *)$, either $\overline e\in lett(\gamma_2)$ or $\overline e\in \text{lk}(lett(\gamma_2))$. To see that $\text{lk}(lett(\gamma_2))\cup B(v\rightarrow *)$ does not separate $\Gamma(W,S)$, observe that otherwise $G$ is not one-ended if $\langle \text{lk}(lett(\gamma_2)) \rangle$ is finite or $(\text{lk}(lett(\gamma_2))\cup B(v\rightarrow *),\text{lk}(lett(\gamma_2)))$ is a virtual factor separator for $\Gamma$ if $\langle \text{lk}(\gamma_2)\rangle$ is infinite.

If $s=t$ and $s\in\text{lk}(lett(\gamma_2))$, then there is a vertex $a\in\Gamma$ adjacent to $s$ with $a\notin \text{lk}(lett(\gamma_2))\cup B(v\rightarrow *)$, since $lett(\gamma_2)$ generates an infinite group and $B(v\rightarrow *)$ does not. If $e$ is the edge between $s$ and $a$, we use the path $e$ followed by $e^{-1}$.

If $s=t$ and $s\notin\text{lk}(lett(\gamma_2))$, then there is a vertex $a\in\Gamma$ adjacent to $s$ with $a\notin \text{lk}(lett(\gamma_2))\cup B(v\rightarrow *)$, else $(\text{lk}(lett(\gamma_2))\cup B(v\rightarrow *),\text{lk}(lett(\gamma_2)))$ is a virtual factor separator for $\Gamma$. We construct the path as before.

If $s\neq t$, $s,t\notin\text{lk}(lett(\gamma_2))$ and no such path exists, then $(\text{lk}(lett(\gamma_2))\cup B(v\rightarrow *),\text{lk}(lett(\gamma_2)))$ is a virtual factor separator for $\Gamma$. Note that if there is an edge $e$ between $s$ and $t$, we use the path $e$, $e^{-1}$, $e$ to satisfy the length two requirement.

If $s\neq t$ and $s\in\text{lk}(lett(\gamma_2))$, then there is a vertex $a\in\Gamma$ adjacent to $s$ with $a\notin \text{lk}(lett(\gamma_2))\cup B(v\rightarrow *)$, since $lett(\gamma_2)$ generates an infinite group and $B(v\rightarrow *)$ does not. Now if $t\in \text{lk}(lett(\gamma_2))\cup B(v\rightarrow *)$ we obtain a $b$ adjacent to $t$ with $b\notin \text{lk}(lett(\gamma_2))\cup B(v\rightarrow *)$ and we have a path between $a$ and $b$ as above (or, if $a=b$, we already had the path), or else we connect $a$ and $t$ as above.
\end{proof}

\begin{remark}
Edge paths in $\Gamma$ of the form $(e, e^{-1})$ and $(e, e^{-1}, e)$ may seem unorthodox, but as in \cite{MRT}, they are combinatorially useful in the filter construction.
\end{remark}

\begin{remark}
Note that $U_1^{x_i}(\lambda)^{-1}$ and $U_2^{x_i}(\lambda)^{-1}$ satisfy the hypotheses of $\gamma_2$ in the previous lemma.
\end{remark}

Recall the filter construction presented near the beginning of this section, and notice that Lemma \ref{avoidlink} gives us more control during the fan construction process: instead of avoiding only $B(v\rightarrow *)$ when choosing paths in $\Gamma(W,S)$ to construct a fan based at $v$, we can avoid $B(v\rightarrow *)$ together with $\text{lk}(lett(\gamma))$, where $\gamma$ could potentially begin the down edge path of a diamond based at $v$. This is the key idea that allows us to keep the Cayley geodesics in our filter ``straight'' (in the CAT(0) sense), which makes the limit set of the filter small in $\partial X$. We'll now specify our choice of $\gamma$ at each vertex $v$ in the filter.

Recall that $W$ acts geometrically on a CAT(0) space $X$ giving a map $C:\Lambda\rightarrow X$ (respecting the action of $W$). The $\Gamma$ geodesics $(\beta, e_{m+1}, e_{m+2}, \dots)$ and $(\beta, d_{m+1}, d_{m+2}, \dots)$  $(c + d)$-track two CAT(0) geodesics in $X$ as in Lemma \ref{sameseg}, and $x_i$ denotes the endpoint of $(\beta, e_{m+1}, \dots, e_i)$, for $i\geq m$.

\begin{defn}
For each vertex $v$ of $\Lambda$, let $\rho_v$ be a $\Lambda$-geodesic from $\ast$ to $v$ such that $C(\rho_v)$ $\delta_1$-tracks the $X$-geodesic from $C(\ast)$ to $C(v)$ (Lemma \ref{deltatrack}).
\end{defn}

\begin{defn}
Suppose $\lambda$ is a geodesic extending $(\beta, e_{m+1}, \dots, e_i)$ for some $i \geq m$, and  $y$ and $z$ are vertices of $\lambda$ with $d(z,*)> d(y,*)=k$. We say $z$ is $R-$\textbf{wide in the} $\tau$ \textbf{direction} at $y$ if the $\Lambda$-distance from $y$ to $\rho_z(k)$ is at least $R$, and the down edge path at $y$ of the diamond for $(\beta, e_{m+1}, \dots, e_i, \lambda)$ and $\rho_z$ can be rearranged to begin with $\tau$. If $z$ is the endpoint of $\lambda$, we say $\lambda$ is $R-$wide in the $\tau$ direction at $y$.
\end{defn}

\begin{remark}
Using the notation in the definition, if $y=x_i$ and $d(\rho_z(i),x_i)\geq14N^2$, then $z$ is $14N^2$-wide in either the $U_1^{x_i}$ or $U_2^{x_i}$ direction at $x_i$, by Lemma \ref{arrangeu2}.
\end{remark}

Let $\delta_0=\text{max}\{1, \delta_1, c+d\}$, where $\delta_1$ is the tracking constant from Lemma \ref{deltatrack}, and $c,d$ are the tracking constants from Lemma \ref{sameseg}. 

Let $\lambda$ be a geodesic extending $(\beta, e_{m+1}, \dots, e_i)$ for some $i \geq m$. Set $A^i=U_1^{x_i}$, and define $A^i(\lambda)$ as follows:
\begin{enumerate}
\item If $U_1^{x_i}(\lambda)=U_2^{x_i}(\lambda)$, then set $A^i(\lambda)=U_1^{x_i}(\lambda)$.
\item If $U_1^{x_i}(\lambda)\neq U_2^{x_i}(\lambda)$ and $\lambda$ is not at least $20N^2\delta_0$ wide in the $U_1^{x_i}$ or $U_2^{x_i}$ direction at $x_i$, then set $A^i(\lambda)=U_1^{x_i}(\lambda)$.
\item If $U_1^{x_i}(\lambda)\neq U_2^{x_i}(\lambda)$ and $\lambda$ is at least $20N^2\delta_0$ wide in the $U_1^{x_i}$ direction at $x_i$ but less than $21N^2\delta_0$ wide in the $U_1^{x_i}$ direction at $x_i$, then set $A^i(\lambda)=U_1^{x_i}(\lambda)$ (and similarly for $U_2^{x_i}$).
\item If $U_1^{x_i}(\lambda)\neq U_2^{x_i}(\lambda)$ and $\lambda$ is at least $21N^2\delta_0$ wide in the $U_1^{x_i}$ direction at $x_i$, then let $\lambda_0$ be the longest initial segment of $\lambda$ such that $\lambda_0$ is at least $20N^2\delta_0$ wide in the $U_1^{x_i}$ direction at $x_i$ but not $21N^2\delta_0$ wide in the $U_1^{x_i}$ direction at $x_i$. Then set $A^i(\lambda)$ to be a shortest geodesic based at the endpoint of $\lambda$ containing an edge in each wall of $U_1^{x_i}(\lambda_0)$ (and similarly for $U_2^{x_i}$). By Lemma \ref{shortback}, $A^i(\lambda)$ geodesically extends  to $\ast$.
\end{enumerate}

At the endpoint of each such $\lambda$, we will construct fans avoiding $\text{lk}(lett(A^i(\lambda)))\cup B((\beta, e_{m+1}, \dots, e_i, \lambda))$ as in Lemma \ref{avoidlink}.

\vspace{5mm}

The next lemma explains why the last step in the above process is significant.

\begin{lemma}
\label{avoidworks}
Let $(\lambda_1, \lambda_2)$ be a geodesic extension of $(\beta, e_{m+1}, \dots, e_i)$. Let $\tau$ be a shortest geodesic based at the endpoint of $\lambda_2$ containing an edge in each wall of $U_1^{x_i}(\lambda_1)$. If $e$ is an edge that geodesically extends $(\beta, e_{m+1}, \dots, e_i, \lambda_1, \lambda_2)$ with $\overline e\notin \text{lk}(lett(\tau))$, then for any geodesic extension $\gamma$ of $(\beta, e_{m+1}, \dots, e_i, \lambda_1, \lambda_2, e)$ and any rearrangement $\gamma'$ of $(\beta, e_{m+1}, \dots, e_i, \lambda_1, \lambda_2, e, \gamma)$, no edge in $w(e)$ can appear on the up edge path at the endpoint of $\lambda_1$ of the diamond for $(\beta, e_{m+1}, \dots, e_i, \lambda_1, \lambda_2, e, \gamma)$ and $\gamma'$ if the down edge path at the endpoint of $\lambda_1$ contains edges in all the walls of $U_1^{x_i}(\lambda_1)$. 
\end{lemma}

\begin{proof}
Suppose not; i.e. there is a geodesic extension $\gamma$ of $(\beta, e_{m+1}, \dots, e_i, \lambda_1, \lambda_2, e)$ and a rearrangement $\gamma'$ of $(\beta, e_{m+1}, \dots, e_i, \lambda_1, \lambda_2, e, \gamma)$ such that an edge $e'$ of $w(e)$ appears on the up edge path at the endpoint of $\lambda_1$ of the diamond for these geodesics, and the down edge path at the endpoint of $\lambda_1$  contains edges in all the walls of $U_1^{x_i}(\lambda_1)$.  Then $w(e')=w(e)$ crosses all walls of $U_1^{x_i}(\lambda_1)$. Let $c_1$ be an edge of $\tau$ such that $\overline e$ does not commute with $\overline c_1 $. In particular, $w(c_1)$ is not a wall of $U_1^{x_i}(\lambda_1)$. By the definition of $\tau$, there is an edge $c_2$ of $\tau$, following $c_1$, such that $\overline c_1$ does not commute with $\overline c_2$. The walls $w(c_2)$ and $w(e)$ are on opposite sides of $w(c_1)$ (see Remark \ref{wallsw}), so they do not cross. In particular, $w(c_2)$ is not a wall of $U_1^{x_i}(\lambda_1)$. Clearly we can continue picking $c_i$ in such a way, but since the length of $\tau$ is finite, this process must stop. This gives the desired contradiction.
\end{proof}

\begin{remark}
\label{onedirectavoid}
Note that Lemma \ref{avoidworks} does not require that $U_1^{x_i}(\lambda_1)\neq U_2^{x_i}(\lambda_1)$ or $U_1^{x_i}((\lambda_1, \lambda_2))\neq U_2^{x_i}((\lambda_1, \lambda_2))$. It is easy to show from our construction that if $U_1^{x_i}(\lambda_1)=U_2^{x_i}(\lambda_1)$, then $\tau$ (as defined in Lemma \ref{avoidworks}) has the same walls as $U_1^{x_i}((\lambda_1, \lambda_2))=U_2^{x_i}((\lambda_1, \lambda_2))$, and so avoiding $\text{lk}(lett(A^i(\lambda)))$ has the effect that no wall of $\lambda_2$ can contain an edge of an up edge path at the end point of $\lambda_1$ for a diamond as described in Lemma \ref{avoidworks}.
\end{remark}

For a geodesic extension $\lambda$ of $(\beta, d_{m+1}, \dots, d_i)$, we define $A_d^i(\lambda)$ in the analagous way. To simplify notation, we will only deal with geodesic extensions $\lambda$ of $(\beta, e_{m+1}, \dots, e_i)$, except where necessary. 

We now return to the filter construction. Set $a_1=\overline e_{m+1}$ and $b_1=\overline d_{m+1}$. We have $a_1,b_1\notin B(x_m\rightarrow *)$, so let $a_1, t_1, \dots, t_k, b_1$ be the vertices of a path of length at least 2 (Lemma \ref{avoidlink}) from $a_1$ to $b_1$ in $\Gamma(W,S)$, where each $t_i \notin \text{lk}(lett(A^m))\cup B(x_m \rightarrow *)$. We construct a fan in $\Lambda$ as before:

\vspace{2mm}

\begin{center}
\includegraphics{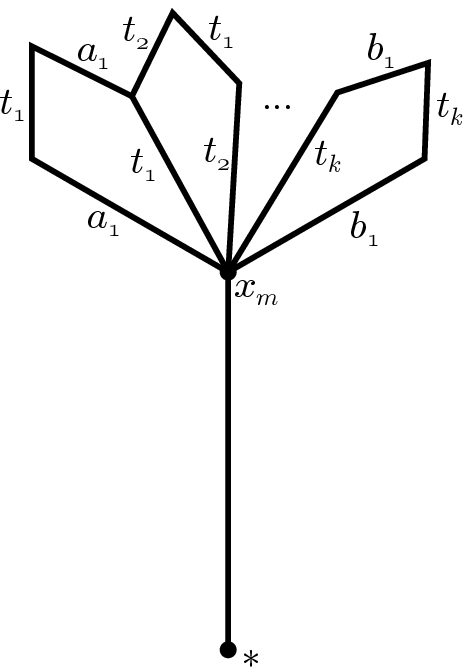}
\end{center}

\centerline{\textsc{Figure 7}}

\vspace{3mm}

\begin{defn}
The edges labeled $a_1$ and $b_1$ at $x_m$ in the fan are called (respectively) the \textbf{left} and \textbf{right fan edges} at $x_m$. The edges labeled $t_1,\dots,t_{k}$ at $x_m$ are called \textbf{interior fan edges}. This fan is called the \textbf{first-level} fan, and the vertices at the endpoints of the edges based at $x_m$ and labeled $x_{m+1}, t_1, \dots, t_k, y_{m+1}$ are called \textbf{first-level} vertices.
\end{defn}

Now, let $a_2=\overline e_{m+2}$, $b_2=\overline d_{m+2}$ and let $w_i$ be the edge at $x_m$ labeled $t_i$ for $1 \leq i \leq k$. Continue constructing the filter by constructing fans avoiding $\text{lk}(lett(A^m((w_i))))\cup B((\beta, w_i))$ at the endpoint of each $w_i$, avoiding $\text{lk}(lett(A^{m+1}))\cup B(x_{m+1} \rightarrow *)$ at $x_{m+1}$, and avoiding $\text{lk}(lett(A_d^{m+1}))\cup B(y_{m+1} \rightarrow *)$ at $y_{m+1}$. Each of these fans is called a \textbf{second-level} fan, and each vertex of distance $2$ from $x_m$ (that will be the base vertex of a third-level fan) is called a \textbf{second-level} vertex. 

\vspace{1mm}

\begin{center}
\includegraphics{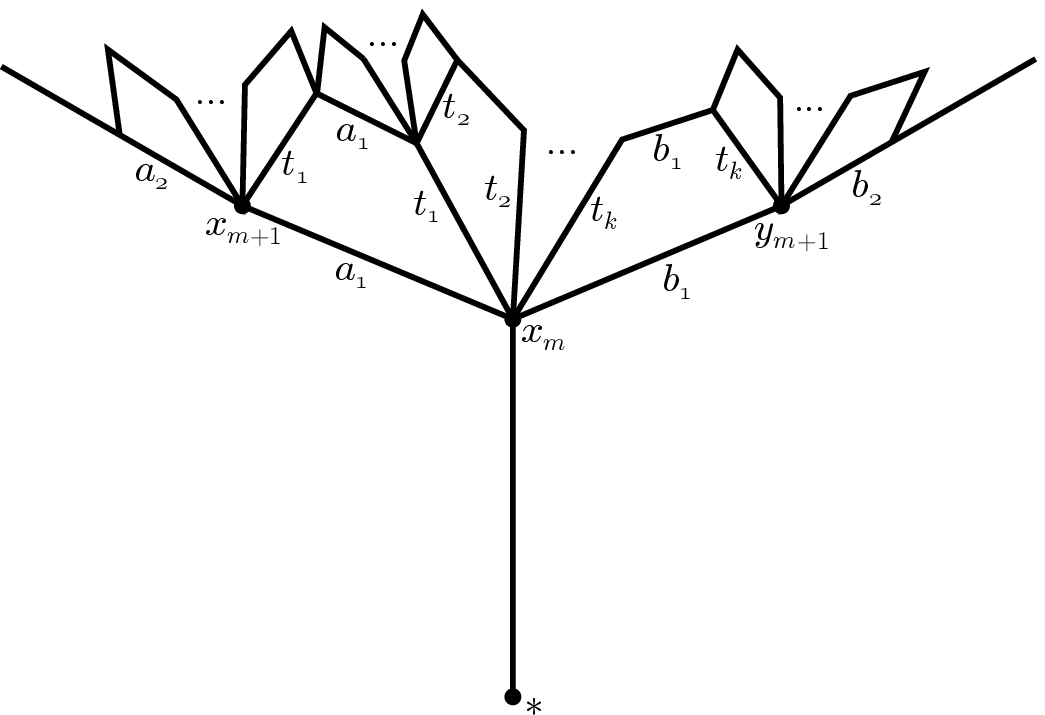}
\end{center}

\centerline{\textsc{Figure 8}}

\vspace{3mm}

It could occur that two edges of this graph share a vertex and are labeled the same; for example, we could have $t_1=a_2$ in Figure 8. We do not identify these edges; instead, we will construct an edge path between them as described in Lemma \ref{avoidlink} and extend the graph between them.

In order to build the third-level fans, we must specify geodesics from $x_m$ to each vertex defined so far, so that $A^i(\lambda)$ is well-defined at each second-level vertex. We'll do this by picking the upper left edge from each first-level fan-loop to be a \textbf{non-tree} edge. This specifies a geodesic from $x_m$ to each second-level vertex. We designate the upper right edge from each second-level fan as a non-tree edge, and continue alternating at each level, so the upper right edge of a $n$-th level fan is a non-tree edge if $n$ is even, and the upper left edge of a $n$-th level fan is a non-tree edge if $n$ is odd. By continuing to construct fans and designate non-tree edges, we construct a filter for our $\Lambda$-geodesics $(\beta, e_{m+1}, e_{m+2}, \dots)$ and $(\beta, d_{m+1}, d_{m+2}, \dots)$. 

Recall that for an edge $a$ of $\Lambda(W,S)$ with initial vertex $y_1$ and terminal vertex $y_2$, an edge $e$ with initial vertex $w_1$ and terminal vertex $w_2$ is in the same wall as $a$ if there is an edge path $(t_1, \dots, t_n)$ in $\Lambda(W,S)$ based at $w_1$ so that $w_1 \overline t_1 \cdots \overline t_n=y_1$ and $w_2 \overline t_1  \cdots \overline t_n=y_2$, and $m(\overline e, \overline t_i )=2$ for each $1\leq i \leq n$. For two edges $a$ and $e$ of $F$, we say $a$ and $e$ are in the same \textbf{filter wall} if there is such a path $(t_1, \dots, t_n)$ in $F$.

\begin{remark}
\label{filterprops}
The following are useful facts about a filter $F$ for two such geodesics ((1)-(5) from \cite{MRT}):

\begin{enumerate}

\item Each vertex $v$ of $F$ has exactly one or two edges beneath it, and there is a unique fan containing all edges (a left and right fan edge, and at least one interior edge) above $v$. We would not have this fact if we allowed association of same-labeled edges at a given vertex.

\item If a vertex of $F$ has exactly one edge below it, then the edge is either $e_i$ (for some $i$), $d_i$ (for some $i$), or an interior fan edge.

\item If a vertex of $F$ has exactly two edges below it, then one is a right fan edge (the one to the left), and one is a left fan edge, and both belong to a single fan loop.

\item $F$ minus all non-tree edges is a tree containing $(\beta, e_{m+1}, e_{m+2}, \dots)$ and $(\beta, d_{m+1}, d_{m+2}, \dots)$ and all interior edges of all fans.

\item If $T$ is the tree obtained from $F$ by removing all non-tree edges, then there are no dead ends in $T$; i.e. for every vertex $v$ of $T$, there is an interior edge extending from $v$.

\item No two consecutive edges of $T$ not on $(\beta, e_{m+1}, e_{m+2}, \dots)$ or $(\beta, d_{m+1}, d_{m+2}, \dots)$ are right (left) fan edges.

\item If $\lambda$ is a geodesic in $F$ extending $(\beta, e_{m+1}, \dots, e_i)$ (and not passing through $x_{i+1}$), then $\lambda$ shares at most one filter wall with $(e_{i+1}, e_{i+2}, \dots)$, and it is the wall of $e_{i+1}$.

\end{enumerate}
\end{remark}

By rescaling, we may assume the image of each edge of $\Lambda$ under $C$ is of length at most 1 in $X$. Then for vertices $v$ and $w$ of $\Lambda$, $d_\Lambda(v,w) \geq d_X(C(v), C(w))$. 

\begin{lemma}
\label{twodirect}
If $(\beta, e_{m+1}, \dots, e_i, \lambda)$ is geodesic in the tree $T$ with endpoint $v$ and $U_1^{x_i}(\lambda)\neq U_2^{x_i}(\lambda)$, then some point on the CAT(0) geodesic between $C(v)$ and $C(*)$ is within $X$-distance $101N^2\delta_0$ of $C(x_i)$.
\end{lemma}

\begin{proof}
Suppose otherwise; then the endpoint $v$ of $\lambda$ is at least $100N^2\delta_0$ wide at $x_i$, and so suppose $v$ is wide in the $U_1^{x_i}$ direction at $x_i$. Choose the last vertex $w$ on $\lambda$ such that $w$ is between $20N^2\delta_0$ and $21N^2\delta_0$ wide in the $U_1^{x_i}$ direction at $x_i$, so that every vertex between $v$ and $w$ on $\lambda$ is at least $21N^2\delta_0$ wide in the $U_1^{x_i}$ direction at $x_i$. Let $\lambda_w$ be the segment of $\lambda$ starting at $x_i$ and ending at $w$. We will show that $v$ is wide in the $U_1^{x_i}(\lambda_w)$ direction at $w$ and that $v$ cannot be wide in the $U_1^{x_i}(\lambda_w)$ direction at $w$, obtaining a contradiction.
\\

\noindent \textbf{Claim 1: } The vertex $v$ is wide in the $U_1^{x_i}(\lambda_w)$ direction at $w$.
\\

\noindent Recall that $\rho_w$ and $\rho_v$ are $\Lambda$-geodesics $\delta_1$-tracking the $X$-geodesics from $C(*)$ to $C(w)$ and $C(v)$ respectively. By CAT(0) geometry, $\rho_v$ is at least $75N^2\delta_0$ wide at $w$, since $w$ is less than $21N^2\delta_0$ wide at $x_i$. Consider Figure 9, with diamonds for these geodesics as in Lemma \ref{diamond}:

\begin{center}
\includegraphics{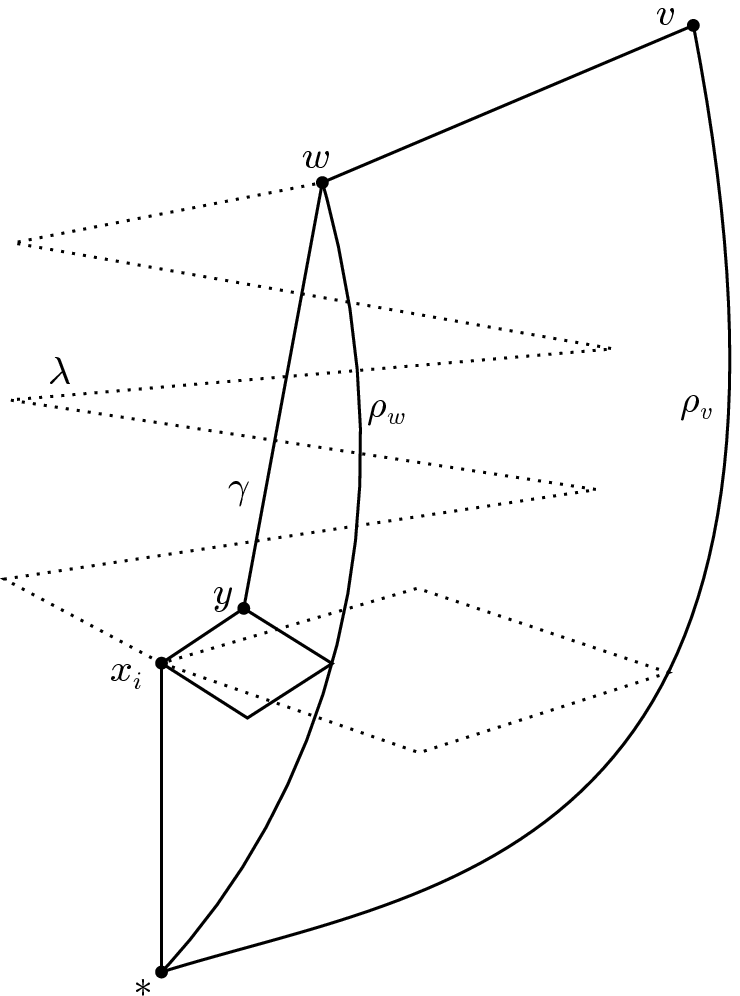}
\end{center}

\centerline{\textsc{Figure 9}}

\vspace{3mm}

Let $y$ be the endpoint of the up edge path of the diamond at $x_i$ for $\rho_w$ and $(\beta, e_{m+1}, \dots, e_i, \lambda_w)$, and let $\gamma_0$ be any geodesic from $y$ to $w$. A simple van Kampen diagram argument shows that there is a rearrangement $\gamma_1$ of $\gamma_0$ such that if $w(c_1), w(c_2), \dots, w(c_n)$ are the walls of the edges of $\gamma_0$, then $\gamma_1$ crosses these walls in the same order as $\rho_w$. Let $\gamma$ be any geodesic from $x_i$ to $y$ followed by $\gamma_1$. By Lemma \ref{diamond}, it is clear that each vertex $x$ of $\gamma$ is of $\Lambda$-distance less than $21N^2\delta_0$ from the corresponding vertex $x'$ of $\rho_w$ (satisfying $d(x, *)=d(x', *)$ in $\Lambda$).  Therefore $\gamma$ is of $\Lambda$-distance at least $54N^2\delta_0$ from $\rho_v$. Now, if no vertex of $\lambda_w$ is within $\Lambda$-distance $14N^2$ of the corresponding vertex of $\rho_v$, then by Lemma \ref{nochange} (with $\lambda_1$ trivial), $v$ is $75N^2\delta_0$ wide in the $U_1^{x_i}(\lambda_w)$ direction at $w$, as claimed. Suppose there are vertices of $\lambda_w$  within $\Lambda$-distance $14N^2$ of the corresponding vertices on $\rho_v$, and list the consecutive points $z_1, \dots ,z_\ell$ of $\lambda_w$ (with $z_1$ closest to $x_i$) where each $z_j$ has the property that if $g_j$ and $m_j$ are the points on $\gamma$ and $\rho_v$ respectively with $d(z_j, *)=d(g_j, *)=d(m_j, *)$, then $\abs{d(z_j, g_j) - d(z_j, m_j)} < N$ (so each $z_j$ is almost $\Lambda$-equidistant from its corresponding points on $\gamma$ and $\rho_v$). 
% These $z_i$ exist since $\tau_w$ starts close to $\gamma'$ and then gets close to $\mu$, so somewhere in between it is equidistance  from both.
Let $\lambda_{z_j}$ denote the initial segment of $\lambda_w$ ending at each $z_j$. Now, $\rho_v$ (equivalently $v$) is wide in the $U_1^{x_i}(\lambda_{z_1})$ direction at $z_1$, since $\lambda_w$ has not yet passed close to $\rho_v$. Now consider the down edge path of the diamond at $z_1$ for $\lambda_w$ and $\gamma$; this path is of length more than $7N^2$ and must have edges in all the walls of $U_2^{x_i}(\lambda_{z_1})$ (Lemma \ref{arrangeu2}), else by Lemma \ref{doublediamond}, $\gamma$ and $\rho_v$ would be close. Now, if $\rho_v$ is wide in the $U_2^{x_i}(\lambda_{z_2})$ direction at $z_2$, then the down edge path at $z_2$ for the diamond for $\lambda_w$ and $\gamma$ must have edges in all the walls of $U_1^{x_i}(\lambda_{z_2})$; however, by Lemma \ref{nochange}, at most one of these directions could have switched, since $\lambda$ does not pass close to one of $\rho_v$ or $\gamma$ between $z_1$ and $z_2$. Continuing this argument along the $z_i$ shows that $v$ is wide in the $U_1^{x_i}(\lambda_w)$ direction at $w$, as claimed.
\\

\noindent \textbf{Claim 2: }The vertex $v$ cannot be wide in the $U_1^{x_i}(\lambda_w)$ direction at $w$.
\\

\noindent Note that no interior fan edges on $\lambda$ between $v$ and $w$ can have walls appearing on the up edge path of a $U_1^{x_i}(\lambda_w)$ diamond at $w$ by Lemma \ref{avoidworks}, since all of these edges have labels chosen to avoid $\text{lk}(lett(U_1^{x_i}((\lambda_w, ...))))$. Also note that if the first edges of $\lambda$ after $\lambda_w$ are a right fan edge followed by a left fan edge, the left fan edge shares a wall with an interior fan edge adjacent to $\lambda$, and so it was also chosen to avoid $\text{lk}(lett(U_1^{x_i}((\lambda_w, ...))))$, and so no edge in its wall can appear on a $U_1^{x_i}(\lambda_w)$ diamond at $w$ (and similarly for a left fan edge followed by right fan edge). The same analysis holds for any right or left fan edge appearing after an interior fan edge (except for at most one edge of $\lambda$, which could share a wall with a right/left fan edge based at $w$). Thus the only way $\lambda$ can have enough edges in the same walls as edges on the up edge path of a $U_1^{x_i}(\lambda_w)$ diamond is if a large sequence of the edges of $\lambda$ immediately after $\lambda_w$ are all right fan edges or all left fan edges, which cannot happen by (6) of Remark \ref{filterprops}. Thus $v$ is not wide in the $U_1^{x_i}(\lambda_w)$ direction at $x_i$, which gives the desired contradiction.
\end{proof}

\begin{lemma}
If $\lambda$ is a geodesic in the tree $T$ with endpoint $v$ that extends $(\beta, e_{m+1}, \dots, e_i)$ and $U_1^{x_i}(\lambda) = U_2^{x_i}(\lambda)$, then some point on the CAT(0) geodesic between $C(v)$ and $C(*)$ is within $X$-distance $118N^2\delta_0$ of $C(x_i)$.
\end{lemma}

\begin{proof}
Let $\lambda_w$ be the shortest initial segment of $\lambda$ such that $U_1^{x_i}(\lambda_w) = U_2^{x_i}(\lambda_w)$, and let $w$ be the endpoint of $\lambda_w$. By Lemma \ref{twodirect}, the CAT(0) geodesic between $C(w)$ and $C(*)$ comes within $X$-distance $101N^2\delta_0$ of $C(x_i)$. Note that if the CAT(0) geodesic between $C(v)$ and $C(*)$ is more than $17N^2\delta_0$ from $C(w)$, then $v$ (equivalently $\lambda$) is at least $16N^2\delta_0$ wide in the $U_1^{x_i}(\lambda_w)$ direction at $w$. When $U_1^{x_i}(\lambda_w) = U_2^{x_i}(\lambda_w)$ we have the following cases:

\noindent \textbf{Case 1:} No geodesic extension of $(\beta, e_{m+1}, \dots, e_i, \lambda_w)$ leads to a bigon $16N^2$ wide at $w$.

In this case, $\lambda$ is not $16N^2$ wide in any direction at $w$, so by CAT(0) geometry, some point on the CAT(0) geodesic between $C(v)$ and $C(*)$ is within $X$-distance $118N^2\delta_0$ of $C(x_i)$..

\noindent \textbf{Case 2:} For any geodesic $\mu$ from $*$ to the endpoint of $(\beta, e_{m+1}, \dots, e_i, \lambda)$, if the bigon determined by $\mu$ and $(\beta, e_{m+1}, \dots, e_i, \lambda)$ is $16N^2$ wide at $w$, then it is wide in the $U_1^{x_i}(\lambda_w)$ direction at $w$.

From Lemma \ref{avoidworks}, Remark \ref{onedirectavoid}, and our filter construction, we know that any interior edge on $\lambda$ after $w$ cannot have its wall on the up edge path of a $U_1(\lambda_w)$ diamond at $w$. If the first edges of $\lambda$ after $w$ are a right fan edge followed by a left fan edge, the left fan edge shares a wall with an interior fan edge adjacent to $\lambda$, and so the left fan edge also cannot have an edge in its wall on the up edge path of a $U_1(\lambda_w)$ diamond at $w$. The same analysis holds for any left or right fan edge following an interior fan edge (except for at most one edge of $\lambda$, which could share a wall with a right/left fan edge based at $w$). Thus by (6) of Remark \ref{filterprops}, $\lambda$ cannot be $16N^2$ wide in the $U_1(\lambda_w)$ direction at $w$, so some point on the CAT(0) geodesic between $C(v)$ and $C(*)$ is within $X$-distance $118N^2\delta_0$ of $C(x_i)$.
\end{proof}

Suppose $X$ is a CAT(0) space, $\ast\in X$ a base point and $B_n(\ast)$  the open $n$-ball about $\ast$. Let $\overline X$ be the compact metric space $X\cup \partial X$. If $F$ is a filter in $X$, let $\overline F$ be  the closure of $F$ in $\overline X$. Since $F$ is connected, $\overline F$ is connected. Since $F$ is one-ended, $\overline F-F$ (the limit set of  $F$)  is contained in $C_n$, a component of $\overline F-B_n(\ast)$, for each $n>0$.  Then $\overline F-F=\cap_{n=1}^\infty C_n$ is the intersection of compact connected subsets of a metric space and so is connected. 

\begin{thm}
Suppose $(W,S)$ is a one-ended right-angled Coxeter system, $\Gamma(W,S)$ contains no visual subgroup isomorphic to $(\Integ_2 \ast \Integ_2)^3$, and $W$ does not visually split as $(\mathbb Z_2\ast\mathbb Z_2)\times A$. Then $W$ has locally connected boundary iff $\Gamma(W,S)$ does not contain a virtual factor separator.
\end{thm}

\begin{proof}
If  $(W,S)$ has a virtual factor separator, then by \cite{MRT}, $W$ has non-locally connected boundary.  Suppose $W$ acts geometrically on a CAT(0) space $X$, and let $r$ be a CAT(0) geodesic ray based at a point $*$ of $X$. Let $\epsilon > 0$ be given. We find $\delta$ such that if $s$ is a geodesic ray within $\delta$ of $r$ in $\partial X$, then our filter for $r$ and $s$ has (connected) limit set of diameter less than $\epsilon$ in $\partial X$. In what follows, the constants $c$ and $d$ are the tracking constants from Lemma \ref{sameseg}, $\delta_1$ is the tracking constant from Lemma \ref{deltatrack}, and $\delta_0 = max\{1, \delta_1, c+d\}$. Recall $C: \Lambda(W,S)\rightarrow X$ $W$-equivariantly, and suppose for simplicity $C(*)=*$. 
Choose $M$ large enough so that for all $m \geq M-c-d$, if $s$ is an $X$-geodesic ray based at $*$ within $120N^2\delta_0$ of $C(\beta(m))$ for any Cayley geodesic $\beta$ that $\delta_0$-tracks $r$, then $r$ and $s$ are within $\epsilon/2$ in $\partial X$. Choose $\delta$ so that if $r$ and $s$ are within $\delta$ in $\partial X$, then $r$ and $s$ satisfy $d(r(M),s(M)) < 1$. Now, if $r$ and $s$ are within $\delta$ in $\partial X$, by Lemma \ref{sameseg}, $r$ and $s$ can be $\delta_0$-tracked by Cayley geodesics $\alpha_r$ and $\alpha_s$  sharing an initial segment of length at least $M-c-d$. Let $m=M-c-d$ and denote the ``split point'' of $\alpha_r$ and $\alpha_s$ by $x_m$, as in the filter construction. Similarly, let $\alpha_r(i)=x_i$ and $\alpha_s(i)=y_i$ for $i\geq m$. By the previous two lemmas, for any vertex $v$ in the filter $F$ for $\alpha_r$ and $\alpha_s$, the $X$-geodesic from $C(v)$ to $*$ passes within $118N^2\delta_0$ of $C(x_i)$ (or $C(y_i)$), where $i\geq m$. By CAT(0) geometry, this geodesic must also pass within $119N^2\delta_0$ of $C(x_m)$. Thus every geodesic ray in the limit set of $C(F)$ is within $\epsilon/2$ of $r$ in $\partial X$, so this set has diameter less than $\epsilon$ in $\partial X$.
\end{proof}

\section{Two Interesting Examples}

Let $(W,S)$ be the (one-ended) right-angled Coxeter system with  presentation graph $\Gamma$ give by Figure 10:

\begin{center}
\includegraphics{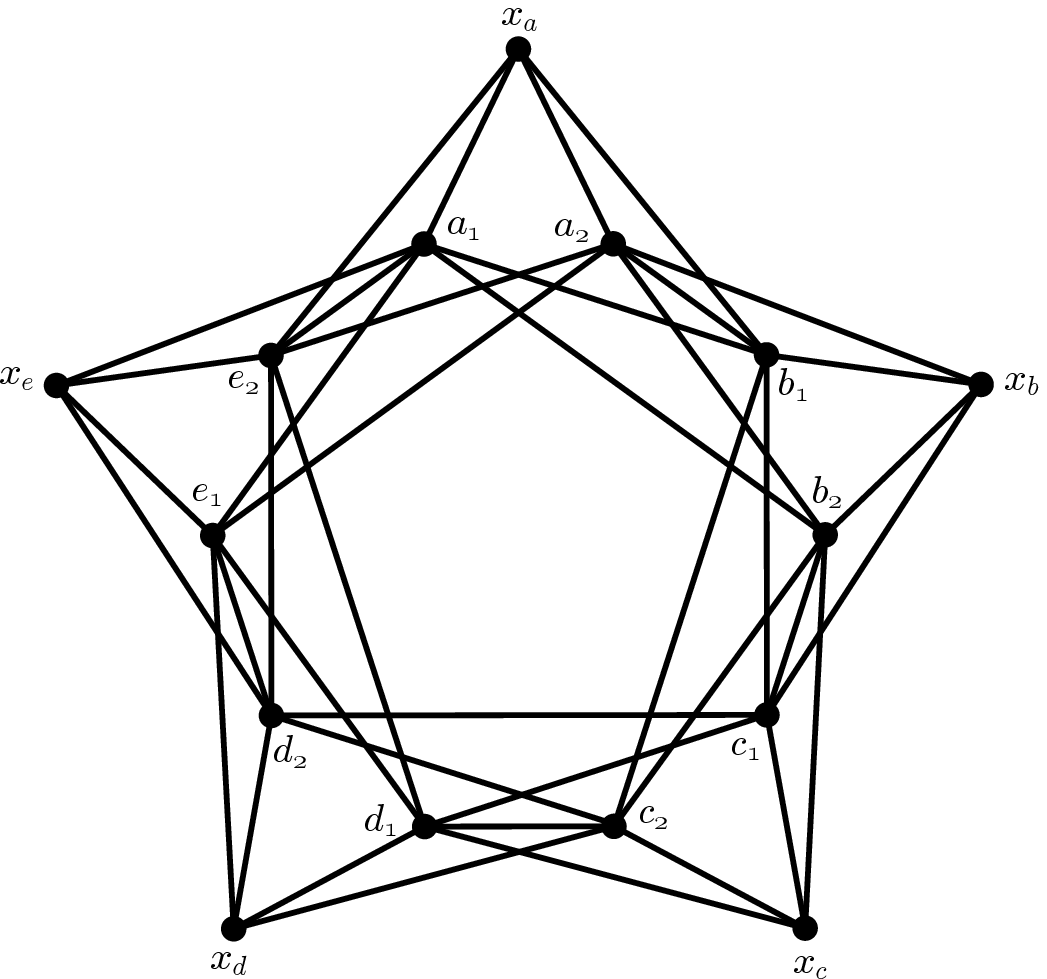}
\end{center}

\centerline{\textsc{Figure 10}}

\vspace{3 mm}

For what follows, let $A=\{a_1,a_2\}$, $B=\{b_1,b_2\}$, $C=\{c_1,c_2\}$, $D=\{d_1,d_2\}$ and $E=\{e_1,e_2\}$.

It is not hard to check that $\Gamma$ has no virtual factor separator, $(W,S)$ does not visually split as a direct product and that $(W,S)$has no visual $(\Integ_2 * \Integ_2)^3$. However, $\Gamma$ contains product separators: for example, $A\cup D$ commutes with $E$, and $A\cup D \cup E$ separates $x_e$ from the rest of $\Gamma$. 

Corollary 5.7 of \cite{MR} gives specific conditions for when the boundary of a right-angled Coxeter group is non-locally connected:
\begin{cor}
Suppose $(W,S)$ is a right-angled Coxeter system. Then $W$ has non-locally connected boundary if there exist $v,w\in S$ with the following properties:
\begin{enumerate}
\item $v$ and $s$ are unrelated in $W$, and
\item $\text{lk}(v)\cap\text{lk}(w)$ separates $\Gamma(W,S)$, with at least one vertex in $S-\text{lk}(v)\cap\text{lk}(w)$ other than $v$ and $w$.
\end{enumerate}
\end{cor}
In particular, they show that if such $v,w$ exist, then $(vw)^{\infty}$ is a point of non-local connectivity in any CAT(0) space acted on geometrically by $W$. Note that if $v,w$ exist as in this corollary, then $(\text{lk}(v)\cap\text{lk}(w),\text{lk}(v)\cap\text{lk}(w))$ is a virtual factor separator for $\Gamma(W,S)$.

Let $G_1=\langle S-x_a \rangle$. Note that $\text{lk}(e_1)\cap\text{lk}(e_2)=A\cup D \cup \{x_e\}$ separates $e_2$ from the rest of  $\Gamma(G_1,S-\{x_a\})$, so $G_1$ has non-locally connected boundary, with $(e_1 e_2)^\infty$ a point of non-local connectivity for $G_1$.
Similarly, let $Q=A\cup B \cup E$ and let $G_2=\langle Q\cup \{x_a\} \rangle$. Then $\text{lk}(e_1)\cap\text{lk}(e_2)=A\cup D \cup \{x_e\}$ separates $e_1$ from the rest of $\Gamma(G_2,Q\cup \{x_a\})$, and so $G_2$ also has non-locally connected boundary, also with $(e_1 e_2)^{\infty}$ a point of non-local connectivity. Note that we now have $W=G_1*_Q G_2$, where $\partial G_1$ and $\partial G_2$ have $(e_1 e_2)^{\infty}$ as a point of non-local connectivity and $Q$ contains $e_1$ and $e_2$, so it would seem that $\partial W$ should also have $(e_1 e_2)^{\infty}$ as a point of non-local connectivity. However, our theorem implies $W$ has locally connected boundary.

For our second example consider the right-angled Coxeter group $(G,S)$ with presentation graph of Figure 11. 

\medskip

\begin{center}
\includegraphics{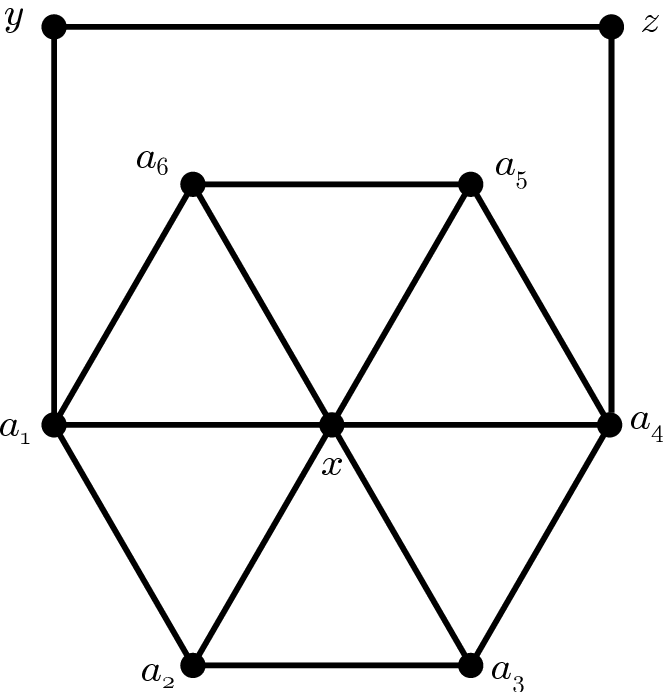}
\end{center}

\centerline{\textsc{Figure 11}}

\vspace{3mm}

Let $A=\{a_1,\ldots , a_6\}$ and $(G', S')$ have the same presentation graph as $(G,S)$ but with each vertex $v$ labeled $v'$. Let $(W,S)$ be the right-angled Coxeter group of the amalgamated product $G\ast_{A=A'}G'$ (where $S=\{x,x', y,y'z,z',A\}$, and $\{x,x'\}$ commutes with $A$). Both $G$ and $G'$ are word hyperbolic and one-ended so they have locally connected boundary. The subgroup $\langle A\rangle$ of $G$ is virtually a hyperbolic surface group and so determines a circle boundary in the boundary of $G$. Still, $W$ has non-locally connected boundary since $(A,A)$ is a virtual factor separator for $(W,S)$.

Aside from being rather paradoxical, these examples show that boundary local connectivity of right-angled Coxeter groups is not accessible through graphs of groups techniques. 

\section{  A Final comment}

If the hypothesis that no $(\mathbb Z_2\ast\mathbb Z_2)^3$ is removed in an attempt to classify all right-angled Coxeter group with locally connected boundary, much of what we develop in this paper carries through. Finitely many directions (as opposed to two) can be defined and used to measure how large the limit set of a filter becomes. As with our development, if the  filter starts to becomes large in a certain direction at a vertex, it is possible to force the subsequent vertices of the filter to slant in any of the other directions.  But when there are only two directions, as in this paper, we are able to show that when we go from being slightly wide in one direction to slightly wide in the other, then the filter did not get too wide in either direction. When there are more than two directions, CAT(0) geometry of right-angled Coxeter groups is not well enough understood yet to accomplish the same thing. It may be that the notion of a slope (or ratio of movement in the various directions) can be developed to overcome this problem.

% ----------------------------------------------------------------

\end{document}